\documentclass[12pt]{article}
\usepackage{eurosym}
\usepackage{amsfonts}
\usepackage{amsmath,amssymb,amsfonts}
\usepackage{amsfonts}
\usepackage{graphicx}
\usepackage{graphics}
\usepackage{url,hyperref}
\usepackage{xcolor}

\newtheorem{theorem}{Theorem}

\newtheorem{corollary}[theorem]{Corollary}

\newtheorem{definition}[theorem]{Definition}
\newtheorem{example}[theorem]{Example}

\newtheorem{lemma}[theorem]{Lemma}

\newtheorem{proposition}[theorem]{Proposition}
\newtheorem{remark}[theorem]{Remark}

\newenvironment{proof}[1][Proof]{\noindent\textbf{#1.} }{\ \rule{0.5em}{0.5em}}

\begin{document}

\title{Infinitesimally helicoidal motions with fixed pitch of oriented
	geodesics of a space form}
	\author{Mateo Anarella and Marcos Salvai
		\thanks{This work was supported by Consejo Nacional de Investigaciones Cient%
			\'{\i}ficas y T\'{e}cnicas and Secretar\'{\i}a de Ciencia y T\'{e}cnica de
			la Universidad Nacional de C\'{o}rdoba.}
}


\date{ }

\maketitle

\begin{abstract}
Let $\mathcal{G}$ be the manifold of all (unparametrized) oriented lines of $%
\mathbb{R}^{3}$. We study the controllability of the control system in $%
\mathcal{G}$ given by the condition that a curve in $\mathcal{G}$ describes
at each instant, at the infinitesimal level, an helicoid with prescribed
angular speed $\alpha $. Actually, we pose the analogous more general
problem by means of a control system on the manifold $\mathcal{G}_{\kappa }$
of all the oriented complete geodesics of the three dimensional space form
of curvature $\kappa $: $\mathbb{R}^{3}$ for $\kappa =0$, $S^{3}$ for $%
\kappa =1$ and hyperbolic 3-space for $\kappa =-1$. We obtain that the
system is controllable if and only if $\alpha ^{2}\neq \kappa $. In the
spherical case with $\alpha =\pm 1$, an admissible curve remains in the set
of fibers of a fixed Hopf fibration of $S^{3}$.

We also address and solve a sort of Kendall's (aka Oxford) problem in this
setting: Finding the minimum number of switches of piecewise continuous
curves joining two arbitrary oriented lines, with pieces in some
distinguished families of admissible curves.
\end{abstract}


\medskip


\section{Introduction}

For $\alpha \in \mathbb{R}$, the helicoid in standard position in $\mathbb{R}%
^{3}$ with angular speed $\alpha $ (or equivalently, with pitch $2\pi
/\alpha $ if $\alpha \neq 0$) is the parametrized surface%
\begin{equation*}
	\phi _{o}:\mathbb{R}^{2}\rightarrow \mathbb{R}^{3}\text{, \ \ \ \ \ \ }\phi
	_{o}\left( s,t\right) =s\cos \left( \alpha t\right) e_{1}+s\sin \left(
	\alpha t\right) e_{2}+te_{3}\text{.}
\end{equation*}%
An helicoid in $\mathbb{R}^{3}$ with angular speed $\alpha $ is a
parametrized surface congruent to $\phi _{o}$ by a rigid transformation of $%
\mathbb{R}^{3}$, that is, a map preserving the distance and the orientation.

Now we state vaguely the problem we are interested in: We fix $\alpha \in 
\mathbb{R}$. Given two oriented straight lines $\ell _{1}$ and $\ell _{2}$
in $\mathbb{R}^{3}$, can we move $\ell _{1}$ to $\ell _{2}$ in such a way
that the swept surface resembles at each instant, at the infinitesimal
level, an helicoid with angular speed $\alpha $?

This is a control problem which does not arise from a linear or affine
linear distribution. Thus, the convenient setting to pose it precisely is
the following, that we learned of from \cite{Agrachev} (see also Subsections
2.1 in \cite{AgrachevG} and 2.6 in \cite{Grong}).

\begin{definition}
	A \textbf{control system} on a smooth manifold $N$ is a fiber subbundle of
	the tangent bundle $TN$, 
	\begin{equation*}
		\begin{array}{lll}
			\mathcal{A} & \overset{\iota }{\longrightarrow } & TN \\ 
			& \searrow & \downarrow \pi \\ 
			&  & N\text{.}%
		\end{array}%
	\end{equation*}
	A smooth curve $\gamma :\left( a,b\right) \rightarrow N$ is said to be 
	\textbf{admissible }if $\gamma ^{\prime }\left( t\right) \in \iota \left( 
	\mathcal{A}\right) $ holds for each $t\in \left( a,b\right) $. A control
	system in $N$ is said to be \textbf{controllable} if for each pair of points
	in $N$ there exists a piecewise admissible curve joining them.
\end{definition}

Let $\mathcal{L}$ be the space of all oriented straight lines of $\mathbb{R}%
^{3}$. This is a four dimensional smooth manifold on which the group of
rigid transformations of $\mathbb{R}^{3}$ acts transitively. The problem
above translates into defining a certain subbundle $\mathcal{A}$ of the
tangent bundle $T\mathcal{L}$. For the sake of generality, we study it for
the three dimensional space forms, that is, we also consider curves in the
manifolds of oriented lines in hyperbolic space $H^{3}$ and of oriented
great circles in the sphere $S^{3}$. We call $\mathcal{G}_{\kappa }$ the
manifold of all oriented geodesics of the three dimensional space form of
curvature $\kappa $, in particular, $\mathcal{L}=\mathcal{G}_{0}$. It is
diffeomorphic to $TS^{2}$ for $\kappa =0,-1$ and to $S^{2}\times S^{2}$ for $%
\kappa =1$.

The fiber bundles involved are not trivial. Since the problem is global,
this is another reason why we choose the above definition of control system.

Our main result, Theorem \ref{TeoP}, asserts the following: For Euclidean
space, the system is controllable if and only if $\alpha \neq 0$. In the
hyperbolic case, the system is controllable for all $\alpha $, while in the
spherical case it is controllable if and only if $\alpha \neq \pm 1$; if $%
\alpha =\pm 1$, an admissible curve consists of great circles in a Hopf
fibration. The precise statement and the proof can be found in Section \ref%
{S2}.

Section \ref{Kendall} addresses a related problem: Given a family $\mathcal{F%
} $ of distinguished curves in a manifold $N$, to find the minimum number of
pieces in $\mathcal{F}$ of continuous curves in $N$ joining two arbitrary
points in $N$, which we call the \textbf{Kendall number} of $\mathcal{F}$.
In fact, this is a problem of the sort David Kendall used to pose to his
students in Oxford in the mid-20th century for the system of a sphere
rolling on the plane without slipping and spinning (that is, $N$ is the five
dimensional manifold of all positions of a sphere resting on a plane) and
the family consists of curves in $N$ determined by rolling along straight
lines. It was solved by John Hammersley in \cite{Oxford}, as a part of a
book dedicated to Kendall for his sixty-fifth birthday (see also Section 4
of Chapter 4 in \cite{Jurd}, where the problem is referred to as the \textbf{%
	\ Oxford problem} and \cite{Biscolla} for a more geometric approach).

In our context we can propose two analogues: for the family $\mathcal{P}%
^{\alpha }$ of curves in $\mathcal{G}_{0}=\mathcal{L}$ describing helicoids
with angular speed $\alpha $, and the family $\mathcal{H}^{\alpha }$ of $%
\alpha $-admissible homogeneous curves in $\mathcal{L}$, that is, those $%
\alpha $-admissible curves
which are orbits of monoparametric groups of rigid transformations. We find
the Kendall numbers for both families.

We would like to thank Yamile Godoy and Eduar\-do Hulett for helpful suggestions.

\section{The $\protect\alpha$-helicoidal control system\label{S2}}

For $\kappa \in\left\{0,1,-1\right\}$, let $M_{\kappa }$ be the space form
of dimension three with constant Gaussian curvature $\kappa $, that is, $%
M_{0}=\mathbb{R}^{3}$, $M_{1}=S^{3}$ and $M_{-1}=H^{3}$. Let $\mathcal{G}%
_{\kappa }$ be the space of all complete oriented geodesics in $M_{\kappa }$
up to parametrizations, i.e., 
\begin{equation*}
	\mathcal{G}_{\kappa }=\{\left[ \sigma \right] \mid \sigma :\mathbb{R}
	\rightarrow M_{\kappa }\text{ is a unit speed geodesic in $M_{\kappa }$}\} 
	\text{,}
\end{equation*}%
where $\sigma _{1}\sim \sigma _{2}$ if $\sigma _{1}\left( t\right) =\sigma
_{2}\left( t+t_{o}\right) $ for all $t$ and some $t_{o}\in \mathbb{R}$.

The isometry group of $M_{\kappa }$ acts transitively on $\mathcal{G}%
_{\kappa }$ and this induces on it a differentiable structure of dimension
four, that renders it diffeomorphic to $TS^{2}$ for $\kappa =0,-1$ \cite{BLP}%
, and $S^{2}\times S^{2}$ for $\kappa =1$ (see Proposition \ref{CS2xS2}).
More precisely, for $\kappa =0$, the map%
\begin{equation}
	\psi :TS^{2}=\left\{ \left( v,u\right) \in S^{2}\times \mathbb{R}^{3}\mid
	u\bot v\right\} \rightarrow \mathcal{L}\text{, \ \ \ \ }\psi \left(
	v,u\right) =\left[ s\mapsto u+sv\right] \text{,}  \label{difeoTS2L}
\end{equation}%
is a diffeomorphism ($v^{\bot }\cong T_{v}S^{2}$). It holds that $\psi ^{-1}%
\left[ s\mapsto u+sv\right] =\left( v,u-\left\langle u,v\right\rangle
v\right) $ (here $u$ is not necessarily orthogonal to $v$).

Before presenting the control system that concerns us, we need the following
definitions. We denote by $\gamma _{v}$ the geodesic in $M_{\kappa }$ with
initial velocity $v$.

\begin{definition}
	\label{philpA}Let $\kappa \in \left\{ 0,1,-1\right\} $ and $\alpha \in 
	\mathbb{R}$. Given $\ell \in \mathcal{G}_{\kappa }$, $p\in \ell $ and a unit
	vector $A\in T_{p}M_{\kappa }$ orthogonal to $\ell $, the $\alpha $\textbf{%
		-helicoidal parametrized surface with initial ray} $\ell $ \textbf{and axis} 
	$\gamma _{A}$, 
	\begin{equation*}
		\phi _{\ell ,p,A}^{\alpha }:\mathbb{R}^{2}\rightarrow M_{\kappa }\text{,}
	\end{equation*}%
	is defined as follows: Suppose that $\ell =\left[ \sigma \right] $ with $%
	\sigma \left( 0\right) =p$ and let $B=A\times \sigma ^{\prime }\left(
	0\right) $. Then 
	\begin{equation}
		\phi _{\ell ,p,A}^{\alpha }\left( s,t\right) =\gamma _{\cos \left( \alpha
			t\right) V_{t}+\sin \left( \alpha t\right) B_{t}}\left( s\right) \text{,}
		\label{phiAlpha}
	\end{equation}%
	where $t\mapsto V_{t}$ and $t\mapsto B_{t}$ are the parallel vector fields
	along $\gamma _{A}$ with initial values $\sigma ^{\prime }\left( 0\right) $
	and $B$, respectively. See Figure \ref{fig:Helicoide}.
\end{definition}

In other words, the axis begins at $p\in \ell $ with initial velocity $A$
perpendicular to $\ell $, and the rays rotate with constant angular speed $%
\alpha $ as they move along the axis with unit speed.

\begin{figure}[ht!]	
	\centerline{
		\includegraphics[width=2in]
		{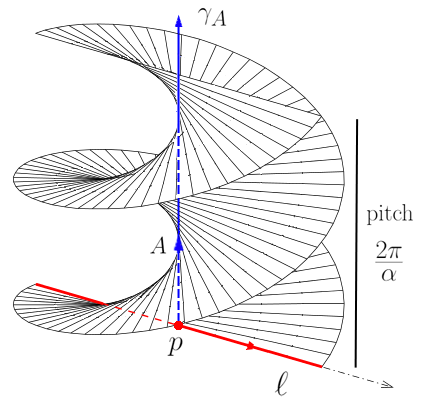}
	}
	\caption{The surface $\phi _{\ell ,p,A}^{\alpha }$ in the Euclidean case}
	\label{fig:Helicoide}
\end{figure}

\begin{definition}
	\label{def}Let $\kappa \in\left\{0,1,-1\right\}$ and $\alpha \in \mathbb{R}$%
	. Given $\ell $, $p$ and $A$ as above, we define the\textbf{\ }$\alpha $%
	\textbf{-helicoidal curve with initial ray }$\ell $ \textbf{and axis }$%
	\gamma _{A}$ as 
	\begin{equation}
		\Gamma _{\ell ,p,A}^{\alpha }:\mathbb{R}\rightarrow \mathcal{G}_{\kappa } 
		\text{,\ \ \ \ \ \ \ \ }\Gamma _{\ell ,p,A}^{\alpha }\left( t\right) =\left[
		s\mapsto \phi _{\ell ,p,A}^{\alpha }\left( s,t\right) \right] \text{,}
		\label{GammaGral}
	\end{equation}
	and the subset $\mathcal{A}_{\kappa }^{\alpha }\subset T\mathcal{G}_{\kappa
	} $ by 
	\begin{equation*}
		\mathcal{A}_{\kappa }^{\alpha }=\left\{ \text{\emph{initial velocities of} }
		\alpha \text{\emph{-helicoidal curves in} }\mathcal{G}_{\kappa }\right\} 
		\text{.}
	\end{equation*}
\end{definition}

We call the elements of this set $\alpha $\textbf{-admissible} tangent
vectors.

Now we are in the position of defining the $\alpha $-helicoidal control
system on $\mathcal{G}_{\kappa }$, that we present in the following
proposition.

\begin{proposition}
	\label{fibrado}Let $\kappa \in\left\{0,1,-1\right\}$ and $\alpha \in \mathbb{%
		R}$. The canonical projection $\mathcal{A} _{\kappa }^{\alpha }\rightarrow 
	\mathcal{G}_{\kappa }$ is a fiber bundle. Moreover, the inclusion $\iota
	_{\kappa }^{\alpha }:\mathcal{A}_{\kappa }^{\alpha }\rightarrow T\mathcal{G}%
	_{\kappa }$ is a fiber subbundle and this gives the control system 
	\begin{equation*}
		\begin{array}{ccc}
			\mathcal{A}_{\kappa }^{\alpha } & \overset{\iota _{\kappa }^{\alpha }}{
				\longrightarrow } & T\mathcal{G}_{\kappa } \\ 
			& \searrow & \downarrow \pi \\ 
			&  & \mathcal{G}_{\kappa }\text{.}%
		\end{array}%
	\end{equation*}
\end{proposition}

We will call the admissible curves of this system 
$\alpha $\textbf{-admissible curves}.

\begin{remark}
	Although the curves $\Gamma _{\ell ,p,A}^{\alpha }$ are orbits of elements
	of $\mathcal{G}_{\kappa }$ under monoparametric groups of isometries of $%
	M_{\kappa }$ \emph{(}conjugate one to another\emph{)}, no vector field on $%
	\mathcal{G}_{\kappa }$ induced by elements of the Lie algebra of \emph{Iso}$%
	\,\left( M_{\kappa }\right) $ is a section of the fiber bundle $\mathcal{A}
	_{\kappa }^{\alpha }\rightarrow \mathcal{G}_{\kappa }$.
\end{remark}

The following proposition reinforces the idea that the problem has a global
nature and suggests the convenience of working in an invariant setting.

\begin{proposition}
	\label{fibratrivial}Let $\kappa \in\left\{0,1,-1\right\}$. If $\alpha
	^{2}\neq \kappa $, the fiber bundle $\mathcal{A}_{\kappa }^{\alpha }$ over $%
	\mathcal{G}_{\kappa }$ is not topologically trivial, that is, the manifold $%
	\mathcal{A}_{\kappa }^{\alpha }$ is not homeomorphic to $\mathcal{G}_{\kappa
	}\times \mathcal{F} _{\kappa }^{\alpha }$, where $\mathcal{F}_{\kappa
	}^{\alpha }$ is the typical fiber of $\mathcal{A}_{\kappa }^{\alpha
	}\rightarrow \mathcal{G} _{\kappa }$.
\end{proposition}

\begin{example}
	\label{Ejemplos} \emph{a)} The curves $\Gamma _{\ell ,p,A}^{\alpha }$, i.e. {%
		\ the }$\alpha $-helicoidal {curves, are clearly }$\alpha $-admissible.
	
	\smallskip
	
	\emph{b)} The homogeneous $\alpha $-admissible curves in the Euclidean case
	are characterized in Proposition \ref{homog}. Among them, the curve of
	straight lines that sweeps a one-sheet hyperboloid is admissible for the
	control system $\left( \iota _{0}^{\alpha },\mathcal{A}_{0}^{\alpha }\right) 
	$, for suitable parameters \emph{(}see the paragraph after that proposition\emph{)}. 
	This also holds for analogous surfaces in $H^{3}$ and $S^{3}$.
	
	\smallskip
	
	\emph{c)} The curve in $\mathcal{L}$ associated with a circular helicoid
	with angular velocity $\alpha \ne 0 $ is not $\alpha $-admissible. We recall
	that this parametrized surface can be built in an analogous manner as $\phi
	_{\ell ,p,A}^{\alpha }$, but taking a unit speed circle $c$ with initial
	velocity $A$, centered at a point on $\ell $, instead of $\gamma _{A}$, and
	using the normal connection of $c$ to rotate $\ell $ along it, with angular
	velocity $\alpha $. See Proposition \ref{HelCirc}.
\end{example}

\medskip

Now we can state our main result. We recall that a submanifold of a vector
space is said to be \textbf{substantial} if is not included in any affine
subspace. The \textbf{Hopf fibrations }of $S^{3}$ are those fibrations by
oriented great circles that are congruent by isometries of the sphere to the
fibration whose fibers are intersections of $S^{3}$ with complex lines,
identifying $\mathbb{R}^{4}\equiv \mathbb{C}^{2}$.

\begin{theorem}
	\label{TeoP}Let $\alpha \in \mathbb{R}$. For $\kappa
	\in\left\{0,1,-1\right\} $, the following assertions are equivalent:
	
	\smallskip
	
	\emph{a)} The control system $\left( \mathcal{A}_{\kappa }^{\alpha },\iota
	_{\kappa }^{\alpha }\right) $ is controllable.
	
	\smallskip
	
	\emph{b)} It holds that $\alpha ^{2}\neq \kappa $.
	
	\smallskip
	
	\emph{c)} For every $\ell \in \mathcal{G}_{\kappa }$, the fiber of $\mathcal{%
		\ A}_{\kappa }^{\alpha }$ over $\ell $ is a substantial submanifold of $%
	T_{\ell }\mathcal{G}_{\kappa }$.
	
	\smallskip
	
	Moreover, in the Euclidean case, the image of a $0$-admissible curve
	consists of parallel straight lines and in the spherical case, if $\alpha
	=\pm 1$, the image of an admissible curve consists of great circles in a
	Hopf fibration.
\end{theorem}

\subsection{Space of oriented geodesics}

We begin by setting some notations for the three dimensional space forms. In
general, we deal with the three cases simultaneously, but the spherical case
will need partly a differentiated approach (see Subsection \ref{SphericalC}).

From now on, $\left\{ e_{0},e_{1},e_{2},e_{3}\right\}$ denotes the canonical
basis of $\mathbb{R}^{4}$. For $\kappa \in\left\{0,1,-1\right\}$, let $%
M_{\kappa }$ be the three dimensional space form with Gaussian curvature $%
\kappa $, that is, $M_{0}=\mathbb{R}^{3}$ and for $\kappa =\pm 1$, $%
M_{\kappa }$ is the connected component of $e_{0}$ of $\left\{ x\in \mathbb{R%
}^{4}:\left\langle x,y\right\rangle _{\kappa }=\kappa \right\} $, where 
\begin{equation}
	\left\langle x,y\right\rangle _{\kappa }=\kappa
	x_{0}y_{0}+x_{1}y_{1}+x_{2}y_{2}+x_{3}y_{3}\text{,}  \label{pik}
\end{equation}%
that induces a Riemannian metric on $M_{\kappa }$. That is, $M_{1}$ is the
sphere $S^{3}$ and $M_{-1}$ is hyperbolic space $H^{3}$. To handle the three
cases simultaneously, sometimes it will be convenient to identify $\mathbb{R}%
^{3}\equiv e_{0}+\mathbb{R}^{3}=\left\{ p\in \mathbb{R}^{4}:p_{0}=1\right\} $%
.

We denote $G_{\kappa }=$ Iso$_{o}\left( M_{\kappa }\right) $, the identity
component of the isometry group on $M_{\kappa }$. Let $O\left( 4\right) $
and $O\left( 1,3\right) $ be the automorphism groups of the inner products $%
\left\langle ,\right\rangle _{1}$ and $\left\langle ,\right\rangle _{-1}$,
respectively. With the identification $\mathbb{R}^{3}\equiv e_{0}+\mathbb{R}%
^{3}$, it holds that%
\begin{align}
	G_{0}& =\left\{ \left( 
	\begin{array}{cc}
		1 & 0 \\ 
		a & A%
	\end{array}
	\right) :a\in \mathbb{R}^{3}\text{, }A\in SO\left( 3\right) \right\} \text{,}
	\label{isomk} \\
	G_{1}& =SO\left( 4\right) =\left\{ A\in O\left( 4\right) :{\mathrm{det}}
	A=1\right\} \text{,}  \notag \\
	G_{-1}& =O_{o}\left( 1,3\right) =\left\{ A\in O\left( 1,3\right) :{\mathrm{%
			det}} A=1 \text{, }\left( Ae_{0}\right) _{0}>0\right\} \text{.}  \notag
\end{align}

Given an orthonormal subset $\left\{ u,v\right\} $ of $T_{p}M_{\kappa }$,
the \textbf{cross product} $u\times v$ is defined as the unique unit vector $%
w$ such that $\left\{ u,v,w\right\} $ is a positively oriented orthogonal
basis of $T_{p}M_{\kappa }$, that is, $\left\{ p,u,v,w\right\} $ is a
positively oriented orthogonal basis of $\left( \mathbb{R}^{4},\left\langle
,\right\rangle _{\kappa }\right) $. For instance, $e_{1}\times e_{2}=e_{3}$.
It can be extended bilinearly to $T_{p}M_{\kappa } \times T_{p}M_{\kappa }$.

Next we recall some properties of the space $\mathcal{G}_{\kappa }$ of
oriented geodesics in $M_{\kappa }$. Their geometry for $\kappa =0,-1$ has
been studied for instance in \cite{GK, SalvaiE, GG, SalvaiH}; for $\kappa =1$
see Subsection \ref{SphericalC}. The isometry group $G_{\kappa }$ acts
transitively on $\mathcal{G}_{\kappa }$ through $g\cdot \left[ \sigma \right]
=\left[ g\circ \sigma \right] $. By abuse of notation, we say that a point $p$
is in $\ell \in \mathcal{G}_{\kappa }$ if for some parametrization $\sigma $
of $\ell $ there exists $s_{o}$ such that $p=\sigma \left( s_{o}\right) $.

We introduce the notation%
\begin{equation*}
	\sin _{1}(r)=\sin r\text{, \ \ }\sin _{0}(r)=r\text{, \ \ }\sin
	_{-1}(r)=\sinh r\text{,\ \ \ }\cos _{\kappa }=\sin _{\kappa }^{\prime }
\end{equation*}%
($\kappa \in \left\{ 0,1,-1\right\} $) and define the geodesic $\sigma _{o}$
in $M_{\kappa }$ and the corresponding element $\ell _{o}$ of $\mathcal{G}%
_{\kappa }$ by%
\begin{equation}
	\sigma _{o}\left( s\right) =\cos _{\kappa }s\ e_{0}+\sin _{\kappa }s\ e_{1}%
	\text{ \ \ \ \ \ \ and \ \ \ \ \ \ }\ell _{o}=\left[ \sigma _{o}\right] 
	\text{.}  \label{sigmacero}
\end{equation}

It will be convenient for us to present $\mathcal{G}_{\kappa }$ explicitly
as a homogeneous space. For $B$, $C\in \mathbb{R}^{2\times 2}$, we denote by
diag$\ \left( B,C\right) \ $the $4\times 4$ matrix with blocks $A$ and $B$
in the main diagonal. We have:

\begin{proposition}
	\emph{\cite{magneticas} }The isotropy subgroup of $G_{\kappa }$ at $\ell
	_{o} $ is $K_{\kappa }=\left\{ k\left( s,t\right) :s,t\in \mathbb{R}\right\} 
	$, where 
	\begin{equation}
		k\left( s,t\right) =\text{\emph{diag}\thinspace }\left( R_{\kappa }\left(
		s\right) ,R_{1}\left( t\right) \right) \text{, \ \ \ \ \ \ with\ \ \ \ \ \ \ 
		}R_{\kappa }\left( t\right) =\left( 
		\begin{array}{cc}
			\cos _{\kappa }t & -\kappa \sin _{\kappa }t \\ 
			\sin _{\kappa }t & \cos _{\kappa }t%
		\end{array}
		\right) \text{.}  \label{isotropiaDeG}
	\end{equation}
\end{proposition}

We consider on $\mathcal{G}_{\kappa }$ the differentiable structure induced
by the bijection%
\begin{equation*}
	F:G_{\kappa }/K_{\kappa }\rightarrow \mathcal{G}_{\kappa }\text{,\ \ \ \ \ \
		\ \ \ \ \ }F\left( gK_{\kappa }\right) =g\cdot \ell_o \text{.}
\end{equation*}

For $\kappa \in\left\{0,1,-1\right\}$ we denote by $\mathfrak{g}_{\kappa }$
the Lie algebra of $G_{\kappa }$. Also from \cite{magneticas} we have%
\begin{equation*}
	\mathfrak{g}_{\kappa }=\left\{ \left( 
	\begin{array}{cc}
		0 & -\kappa x^{T} \\ 
		x & B%
	\end{array}
	\right) :x\in \mathbb{R}^{3}\text{, }B^{T}=-B\right\} \text{.}
\end{equation*}

The Lie algebra of $K_{\kappa }$ is 
\begin{equation*}
	\mathfrak{k}_{\kappa }=\left\{ \text{diag\thinspace }\left( \left( 
	\begin{array}{cc}
		0 & -\kappa t \\ 
		t & 0%
	\end{array}
	\right) ,\left( 
	\begin{array}{cc}
		0 & -s \\ 
		s & 0%
	\end{array}
	\right) \right) :s,t\in \mathbb{R}\right\} \text{.}
\end{equation*}

For column vectors $x,y\in \mathbb{R}^{2}$ we call%
\begin{equation}
	Z(x,y)=\left( 
	\begin{array}{cc}
		0_{2} & (-\kappa x,-y)^{T} \\ 
		(x,y) & 0_{2}%
	\end{array}
	\right) \text{.}  \label{zeta}
\end{equation}

The subspace $\mathfrak{p}_{\kappa }=\left\{ Z(x,y)\in \mathfrak{g}_{\kappa
}:x,y\in \mathbb{R}^{2}\right\} $ of $\mathfrak{g}_{\kappa }$ is an Ad $%
\left( K_{\kappa }\right) $-invariant complement of $\mathfrak{k}_{\kappa }$
and there exists a natural identification 
\begin{equation}
	\left. d\left( F\circ \ \varpi \right) _{I}\right\vert _{\mathfrak{p}%
		_{\kappa }}:\mathfrak{p}_{\kappa }\rightarrow T_{\ell _{o}}\mathcal{G}%
	_{\kappa }\text{,}  \label{identif}
\end{equation}%
where $\varpi :G_{\kappa }\rightarrow G_{\kappa }/K_{\kappa }$ is the
canonical projection.

\subsection{The fiber bundle $\mathcal{A}_{\protect\kappa }^{\protect\alpha %
	}\rightarrow \mathcal{G}_{\protect\kappa }$}

Now we consider a particular case of $\alpha $-helicoidal curve as in (\ref%
{GammaGral}), in good position. Let $\sigma _{o}$ and $\ell _{o}$ be as in (%
\ref{sigmacero}) and let 
\begin{equation}
	p_{o}=e_{0}=\sigma _{o}\left( 0\right) \text{,\ \ \ \ \ \ }A_{o}=e_{3}\text{%
		\ \ \ \ \ \ and\ \ \ \ \ \ }B_{o}=A_{o}\times \sigma _{o}^{\prime }\left(
	0\right) =e_{2}\text{. }  \label{p0A0B0}
\end{equation}%
We call $\Gamma _{o}^{\alpha }$ the curve in $\mathcal{G}_{\kappa }$ defined
by 
\begin{equation}
	\Gamma _{o}^{\alpha }=\Gamma _{\ell _{o},p_{o},A_{o}}^{\alpha }
	\label{gamcer}
\end{equation}%
and denote by $X_{\alpha }$ its initial velocity, that is, 
\begin{equation}
	X_{\alpha }=\left. \tfrac{d}{dt}\right\vert _{0}\ \Gamma _{o}^{\alpha
	}\left( t\right) \in T_{\ell _{o}}\mathcal{G}_{\kappa }\text{.}
	\label{Valpha}
\end{equation}

\begin{proof}[Proof of Proposition \protect\ref{fibrado}]
	Since $G_{\kappa }$ acts transitively on the positively oriented orthonormal
	frame bundle of $M_{\kappa }$, given $\ell$, $p$, $A$ as in Definition \ref%
	{def}, there exists $g\in G_{\kappa }$ such that $g\left( e_{0}\right) =p$, $%
	dg_{e_{0}}\left( e_{3}\right) =A$ and sends $\ell _{o}$ to $\ell $ keeping
	the orientation. Since clearly $G_{\kappa }$ carries $\alpha $-helicodal
	curves in $\alpha $-helicoidal curves, it turns out that the group $%
	G_{\kappa }$ acts transitively on $\mathcal{A}_{\kappa }^{\alpha }$. Thus, $%
	\mathcal{A}_{\kappa }^{\alpha }=\left\{ dg_{\ell _{o}}\left( X_{\alpha
	}\right) :g\in G_{\kappa }\right\} $; in other words, it is the orbit of $%
	X_{\alpha }$ in $T\mathcal{G}_{\kappa }$ under the action of $G_{\kappa }$
	and therefore the inclusion is a fiber subbundle of $T\mathcal{G}_{\kappa }$.
\end{proof}

Next we give an explicit homogeneous presentation of $\mathcal{A}_{\kappa
}^{\alpha }$. We call%
\begin{equation}
	\xi _{\alpha }=\left( 
	\begin{array}{cc}
		0 & -\left( a_{\kappa }^{\alpha }\right) ^{T} \\ 
		a_{1}^{\alpha } & 0%
	\end{array}%
	\right) =Z\left( 
	\begin{array}{cc}
		0 & \alpha \\ 
		1 & 0%
	\end{array}%
	\right) \in \mathfrak{p}_{\kappa }\text{,}  \label{xiAlpha}
\end{equation}%
where $Z$ was defined in (\ref{zeta}) and $a_{\kappa }^{\alpha }=\left( 
\begin{array}{cc}
	0 & \alpha \\ 
	\kappa & 0%
\end{array}%
\right) $.

\medskip

\begin{lemma}
	\label{Vv} Let $X_{\alpha }$ be as in \emph{(\ref{Valpha})}. Then $d\left(
	F\circ \ \varpi \right) _{I}\left( \xi _{\alpha }\right) =X_{\alpha }$.
\end{lemma}

\begin{proof}
	For any $t\in \mathbb{R}$, let $S_{t}\in G_{\kappa }$ given by 
	\begin{equation}
		S_{t}=\left( 
		\begin{array}{cccc}
			\cos _{\kappa }t & 0 & 0 & -\kappa \sin _{\kappa }t \\ 
			0 & \cos \alpha t & -\sin \alpha t & 0 \\ 
			0 & \sin \alpha t & \cos \alpha t & 0 \\ 
			\sin _{\kappa }t & 0 & 0 & \cos _{\kappa }t%
		\end{array}
		\right) \text{.}  \label{St}
	\end{equation}
	Then $S_{t}=\exp \left( t\xi _{\alpha }\right) $, since $S_{s+t}=S_{s}\circ
	S_{t}$ for all $s,t$ and $S_{0}^{\prime }=\xi _{\alpha }$.
	
	Now we check that $S_{t}\sigma _{o}\left( s\right) =\phi _{\ell
		_{o},p_{o},A_{o}}^{\alpha }\left( s,t\right) $ holds for all $s,t\in \mathbb{%
		\ R}$. We fix $t$ and verify that both expressions are equal as functions of 
	$s $. Since they are geodesics with the same initial value $\cos _{\kappa
	}t\ e_{0}+\sin _{\kappa }t\ e_{3}$, it suffices to see that they have the
	same initial velocity. We compute 
	\begin{equation*}
		\left. \tfrac{d}{ds}\right\vert _{0}S_{t}\sigma _{o}\left( s\right)
		=S_{t}\left. \tfrac{d}{ds}\right\vert _{0}\sigma _{o}\left( s\right)
		=S_{t}e_{1}=\cos \left( \alpha t\right) e_{1}+\sin \left( \alpha t\right)
		e_{2}\text{,}
	\end{equation*}
	which coincides with 
	\begin{equation*}
		\left. \tfrac{d}{ds}\right\vert _{0}\phi _{\ell _{o},p_{o},A_{o}}^{\alpha
		}\left( s,t\right) =\left. \tfrac{d}{ds}\right\vert _{0}\gamma _{\cos \left(
			\alpha t\right) V_{t}+\sin \left( \alpha t\right) B_{t}}\left( s\right)
		=\cos \left( \alpha t\right) V_{t}+\sin \left( \alpha t\right) B_{t}\text{,}
	\end{equation*}
	as desired. Finally, 
	\begin{equation*}
		d\left( F\circ \ \varpi \right) _{I}\left( \xi _{\alpha }\right) =d\left(
		F\circ \ \varpi \right) _{I}\left( S_{0}^{\prime }\right) =\left. \tfrac{d}{
			dt}\right\vert _{0}F\circ \ \varpi \circ S_{t}=\left. \tfrac{d}{dt}
		\right\vert _{0}S_{t}\left[ \sigma _{o}\right] \text{,}
	\end{equation*}
	which equals $\left. \frac{d}{dt}\right\vert _{0}\ \Gamma _{o}^{\alpha
	}\left( t\right) =X_{\alpha }$ by the computation above.
\end{proof}


\begin{proposition}
	\label{hachek}\emph{a)} Under the identification \emph{(}\ref{identif}\emph{)%
	}, the fiber of $\mathcal{A}_{\kappa }^{\alpha }$ over $\ell _{0}$ is%
	\begin{equation*}
		\text{\emph{Ad}}\left( K_{\kappa }\right) \left( \xi _{\alpha }\right)
		=\left\{ \text{\emph{Ad}}\left( k\left( s,t\right) \right) \left( \xi
		_{\alpha }\right) :s,t\in \mathbb{R}\right\} \text{,}
	\end{equation*}%
	with $k\left( s,t\right) $ as in \emph{(}\ref{isotropiaDeG}\emph{)}.
	
	\emph{b)} If $v\in \mathcal{A}_{\kappa }^{\alpha }$, then $-v\in \mathcal{A}%
	_{\kappa }^{\alpha }$.
	
	\emph{c)} For $\kappa \in \left\{ 0,-1\right\} $ and $\alpha \neq 0$, $%
	G_{\kappa }$ acts simply transitively on $\mathcal{A}_{\kappa }^{\alpha }$.
\end{proposition}

\begin{proof}
	a) We know from the proof of Proposition \ref{fibrado} that $G_{\kappa }$
	acts transitively on $\mathcal{A}_{\kappa }^{\alpha }$ via the differential.
	Hence, the fiber of $\mathcal{A}_{\kappa }^{\alpha }$ over $\ell _{o}$
	equals $\left\{ dk_{\ell _{o}}\left( X_{\alpha }\right) :k\in K_{\kappa
	}\right\} $. The assertion follows now from the lemma above and the
	commutativity of the diagram 
	\begin{equation}
		\begin{array}{ccc}
			\mathfrak{p}_{\kappa } & \overset{\text{Ad }\left( k\right) }{%
				\longrightarrow } & \mathfrak{p}_{\kappa } \\ 
			\left. d\left( F\circ \ \varpi \right) _{I}\right\vert _{\mathfrak{p}%
				_{\kappa }}\downarrow \ \text{\ \ \ \ \ \ \ \ \ \ \ \ \ \ \ \ \ } &  & \text{
				\ \ \ \ \ \ \ \ \ \ \ \ \ \ \ \ \ }\downarrow \left. d\left( F\circ \ \varpi
			\right) _{I}\right\vert _{\mathfrak{p}_{\kappa }} \\ 
			T_{\ell _{o}}\mathcal{G}_{\kappa } & \overset{dk_{p}}{\longrightarrow } & 
			T_{\ell _{o}}\mathcal{G}_{\kappa }\text{.}%
		\end{array}
		\label{conm2}
	\end{equation}
	
	b) By homogeneity, we may suppose that $v$ is in the fiber over $\ell _{o}$.
	Hence $v$ has the form%
	\begin{equation}
		\text{Ad$\ $}\left( k\left( s,t\right) \right) \left( \xi _{\alpha }\right)
		=\left( 
		\begin{array}{cc}
			0_{2} & -R_{\kappa }\left( t\right) \left( a_{\kappa }^{\alpha }\right)
			^{T}R_{1}\left( -s\right) \\ 
			R_{1}\left( s\right) a_{1}^{\alpha }R_{\kappa }\left( -t\right) & 0_{2}%
		\end{array}%
		\right) \text{.}  \label{Adk}
	\end{equation}%
	Since $R_{1}\left( s+\pi \right) =R_{1}\left( \pi \right) R_{1}\left(
	s\right) =-R_{1}\left( s\right) $, we have that $-v=~$Ad$\ $$\left( k\left(
	s+\pi ,t\right) \right) \left( \xi _{\alpha }\right) $ and so it belongs to
	the fiber over $\ell _{o}$.
	
	\smallskip
	
	c) Let $H_{\kappa }\left( \alpha \right) $ be the isotropy subgroup at $%
	X_{\alpha }$ of the action of $G_{\kappa }$ on $\mathcal{A}_{\kappa
	}^{\alpha }$ (in particular, $H_{\kappa }\left( \alpha \right) \subset
	K_{\kappa }$). We have that that $H_{\kappa }\left( \alpha \right) =\left\{
	k\in K_{\kappa }\mid dk_{\ell _{o}}X_{\alpha }=X_{\alpha }\right\} $, which
	by the diagram (\ref{conm2}) equals 
	\begin{equation*}
		\left\{ k\in K_{\kappa }\mid \text{Ad}\left( k\right) \left( \xi _{\alpha
		}\right) =k\xi _{\alpha }k^{-1}=\xi _{\alpha }\right\} \text{.}
	\end{equation*}
	
	Now, by (\ref{Adk}), $k\left( s,t\right) \in K_{\kappa }$ commutes with $\xi
	_{\alpha }$ if and only if $R_{1}\left( s\right) a_{1}^{\alpha
	}=a_{1}^{\alpha }R_{\kappa }\left( t\right) $, that is, 
	\begin{equation*}
		\left( 
		\begin{array}{cc}
			-\sin s & \alpha \cos s \\ 
			\cos s & \alpha \sin s%
		\end{array}%
		\right) =\left( 
		\begin{array}{cc}
			\alpha \sin _{\kappa }t & \alpha \cos _{\kappa }t \\ 
			\cos _{\kappa }t & -\kappa \sin _{\kappa }t%
		\end{array}%
		\right) \text{.}
	\end{equation*}%
	Therefore, $k\left( s,t\right) \in H_{\kappa }\left( \alpha \right) $ if and
	only if 
	\begin{equation*}
		-\sin s=\alpha \sin _{\kappa }t\text{,\ \ \ \ \ }\cos s=\cos _{\kappa }t%
		\text{\ \ \ \ \ \ and\ \ \ \ \ \ \ }\alpha \sin s=-\kappa \sin _{\kappa }t%
		\text{.}
	\end{equation*}%
	If $\kappa =0$, this implies that $\cos s=1$ and $-\sin s=\alpha t$, and so $%
	R_{\kappa }\left( t\right) =R_{1}\left( s\right) =I$. If $\kappa =-1$, we
	have that $\cos s=\cosh t=1$ and so we arrive at the same conclusion. In
	both cases, $H_{\kappa }\left( \alpha \right) =\left\{ I\right\} $, as
	desired.
\end{proof}

\subsection{The $\protect\alpha $-helicoidal control system in the spherical
	case\label{SphericalC}}

Let $\mathbb{H}$ be the skew field of quaternions. We consider the sphere $%
S^{3}$ as the set of unit quaternions, that is, $S^{3}=\left\{ q\in \mathbb{H%
}\mid \left\vert q\right\vert =1\right\} $, which is a Lie group. It is well
known that, identifying $\mathbb{R}^{4}$ with $\mathbb{H}$, the maps $%
f:S^{3}\longrightarrow SO\left( 3\right) $ and $F:S^{3}\times S^{3}\mapsto
SO\left( 4\right) $ given by%
\begin{equation}
	f\left( p\right) \left( x\right) =px\overline{p}\text{\ \ \ \ \ \ \ \ \ and
		\ \ \ \ \ \ \ \ }F\left( p,q\right) \left( y\right) =py\overline{q}\text{,}
	\label{fyF}
\end{equation}%
for $x\in \operatorname{Im}\left( \mathbb{H}\right) \cong \mathbb{R}^{3}$ and $y\in 
\mathbb{H}\cong \mathbb{R}^{4}$, are both surjective two-to-one morphisms.

For brevity, we call $\mathcal{C}=\mathcal{G}_{1}$ the manifold of all
oriented great circles of $S^{3}$. We have that $S^{3}\times S^{3}$ acts
transitively on $\mathcal{C}$, since the action of $SO\left( 4\right) $ on $%
\mathcal{C}$ is transitive.

It is well known, for instance from \cite{GW} and \cite{Morgan}, that $%
\mathcal{C}$ is diffeomorphic to $S^{2}\times S^{2}$. We include this
assertion in the next proposition and write down the proof since it is
different from the ones given in those articles and shorter; also, it
contributes to establish the nomenclature used later. Note that $S^{1}=\left\{
e^{it}\mid t\in \mathbb{R}\right\} \subset S^{3}$.

\begin{proposition}
	\label{CS2xS2}The transitive action of $S^{3}\times S^{3}$ on $\mathcal{C}$
	has $S^{1}\times S^{1}$ as its isotropy subgroup at $c_{o}=\left[ s\mapsto
	e^{is}\right] $ and induces the \emph{(}well defined\emph{)} diffeomorphism 
	\begin{equation}
		\Phi :\mathcal{C}\rightarrow S^{2}\times S^{2}\text{,\ \ \ \ \ \ \ }\Phi
		\left( \left( p,q\right) \cdot c_{o}\right) =\left( f(p)(i),f(q)(i)\right) 
		\text{.}  \label{Phi}
	\end{equation}
\end{proposition}

\begin{proof}
	Let $\left( p,q\right) \in S^{1}\times S^{1}$. Then $p=e^{it}$ and $q=e^{ir}$
	for some $t,r\in \mathbb{R}$. Thus, $s\mapsto pe^{is}\overline{q}
	=e^{i(s+t-r)}$ belongs to the equivalence class $\left[ s\mapsto e^{is} %
	\right] $ and so $S^{1}\times S^{1}$ is included in the isotropy subgroup.
	Now, we check the other inclusion. Let $p,q\in S^{3}$ such that $pe^{is} 
	\overline{q}=e^{i\left( s+s_{o}\right) }$ for some $s_{o}$ and all $s$.
	Then, $pe^{is}=e^{is}e^{is_{o}}q$ for all $s$ and in particular, $%
	p=e^{is_{o}}q$. Differentiating, we have $pie^{is}=ie^{is}e^{is_{o}}q$ and
	so, $pi=ie^{is_{o}}q=ip$. Since $p$ commutes with $i$, then $p\in S^{1}$ and
	so $q=e^{-is_{o}}p\in S^{1}$ as well. Therefore, the isotropy subgroup at $%
	c_{o}$ is $S^{1}\times S^{1}$.
	
	Now, $\left( S^{3}\times S^{3}\right) /\left( S^{1}\times S^{1}\right) $ is
	canonically diffeomorphic to $\left( S^{3}/S^{1}\right) \times \left(
	S^{3}/S^{1}\right) $. Then the expression for $\Phi $ follows from the fact
	that the morphism $f$ in (\ref{fyF}) induces a transitive action of $S^{3}$
	on $S^{2}\subset \operatorname{Im}\mathbb{H}$, given by $\left( p,u\right) \mapsto
	pu \bar{p} $, with isotropy subgroup at $i$ equal to $S^{1}$.
\end{proof}

Now, we describe in terms of the identification $\Phi$ above the curve $%
\Gamma _{o}^{\alpha }$ in $\mathcal{C}$ in good position defined in (\ref%
{gamcer}). Given $\beta ,\tau \in \mathbb{R}$, we define the isometries 
\begin{equation*}
	R_{\beta }\left( q\right) =e^{\beta k/2}qe^{-\beta k/2}\text{ \ \ \ \ \ and\
		\ \ \ \ \ }T_{\tau }\left( q\right) =e^{\tau k/2}qe^{\tau k/2}
\end{equation*}
of $S^{3}$ (see (\ref{fyF})). The former is a rotation of $\mathbb{R}^{4}$
fixing $1$ and $k$, and rotating the \mbox{$i$-$j$} plane through the angle $%
\beta $. The latter is a transvection in $\tau $ along $t\mapsto e^{tk}$
(i.e. $T_{\tau }\left( e^{tk}\right) =e^{\left( t+\tau \right) k}$ and its
differential realizes the parallel transport along $t\mapsto e^{tk}$, see
for instance Theorem 2 (3) in Note 7 of \cite{KN}). Notice that $R_{\beta }$
and $T_{\tau }$ commute.

\begin{proposition}
	\label{CosasEsf}\emph{a)} The $\alpha $-helicoidal surface in $S^{3}$ with
	axis $t\mapsto e^{tk}$ and initial circle $s\mapsto e^{si}$ is given by $%
	\phi _{o}\left( s,t\right) =T_{t}R_{\alpha t}\left( e^{si}\right) $.
	
	\smallskip
	
	\emph{b)} For the corresponding curve $\Gamma _{o}^{\alpha }$ in $\mathcal{C}
	$, the associated curve in $S^{2}\times S^{2}$ is 
	\begin{equation}
		\left( \Phi \circ \Gamma _{o}^{\alpha }\right) \left( t\right) =\left(
		R_{t\left( 1+\alpha \right) }\left( i\right) ,R_{t\left( 1-\alpha \right)
		}\left( i\right) \right) \text{.}  \label{PhiGamma}
	\end{equation}
	In particular, 
	\begin{equation}
		\left( \Phi \circ \Gamma _{o}^{\alpha }\right) ^{\prime }\left( 0\right)
		=\left( \left( 1+\alpha \right) j,\left( 1-\alpha \right) j\right) \in
		T_{\left( i,i\right) }\left( S^{2}\times S^{2}\right) \text{.}
		\label{GammaPrima}
	\end{equation}
	
	\emph{c)} The fiber of $\mathcal{A}_{1}^{\alpha }$ over $\left( x,y\right)
	\in S^{2}\times S^{2}$, via the identification $\Phi$, is given by 
	\begin{equation}
		\left\{ \left( \left( 1+\alpha \right) z,\left( 1-\alpha \right) w\right)
		:z,w\in \operatorname{Im}\mathbb{H}\text{, }\left\vert z\right\vert =\left\vert
		w\right\vert =1\text{, }z\bot x\text{ and }w\bot y\right\} \text{.}
		\label{fibrak1}
	\end{equation}
\end{proposition}

\begin{proof}
	The first assertion follows from the properties of $R_{\beta }$ and $T_{\tau
	}$ we mentioned when we introduced them above. It implies that 
	\begin{equation*}
		\Gamma _{o}^{\alpha }\left( t\right) =\left[ s\mapsto T_{t}R_{\alpha
			t}\left( e^{si}\right) \right] =\left[ s\mapsto e^{tk/2}e^{\alpha
			tk/2}e^{si}e^{-\alpha tk/2}e^{tk/2}\right] =\left( p_{t},q_{t}\right) \cdot
		c_{o} \text{,}
	\end{equation*}
	where $p_{t}=e^{\left( 1+\alpha \right) tk/2}$, $q_{t}=e^{\left( 1-\alpha
		\right) tk/2}$. Then, $\left( \Phi \circ \Gamma _{o}^{\alpha }\right) \left(
	t\right) =\left( p_{t}i\overline{p_{t}},q_{t}i\overline{q_{t}}\right) $ and
	( \ref{PhiGamma}) follows. A straightforward computation yields (\ref%
	{GammaPrima}).
	
	Now we verify (c). By homogeneity we may suppose $x=y=i$. As we saw in the
	proof of Proposition \ref{fibrado}, the group $G_{1}=SO\left( 4\right) $
	acts transitively on $\mathcal{A}_{1}^{\alpha }$. Since $S^{3}\times S^{3}$
	covers $SO\left( 4\right) $, we may write $\mathcal{A}_{1}^{\alpha }=\left\{
	p\Gamma _{o}^{\prime }\left( 0\right) \overline{q}:p,q\in S^{3}\right\}$.
	
	By Proposition \ref{CS2xS2}, the isotropy subgroup of the action of $%
	S^{3}\times S^{3}$ on $S^{2}\times S^{2}\simeq \mathcal{C}$ is $S^{1}\times
	S^{1}$. Thus, the fiber of $\mathcal{A}_{1}^{\alpha }$ over $c_{o}\simeq
	\left( i,i\right) $ is $\left\{ p\Gamma _{o}^{\prime }\left( 0\right) 
	\overline{q}:p,q\in S^{1}\right\} $ and using (\ref{GammaPrima}) we get that
	it equals 
	\begin{equation*}
		\left\{ \left( \left( 1+\alpha \right) pj\overline{p},\left( 1-\alpha
		\right) qj\overline{q}\right) :p,q\in S^{1}\right\} \text{.}
	\end{equation*}
	Now, putting $p=e^{it}$, we have $pj\overline{p}=e^{it}je^{-it}=\cos \left(
	2t\right) j+\sin \left( 2t\right) k,$ which are exactly the unit elements on 
	$\operatorname{Im}\mathbb{H}$ orthogonal to $i$. Thus, (\ref{fibrak1}) follows.
\end{proof}


Next we recall the concept of Hopf fibration. The left multiplication by $i$
in $\mathbb{H}$ induces on it a vector space structure over $\mathbb{C}$. We
have that $\left\{ \mathbb{C}q\cap S^{3}\mid q\in S^{3}\right\}$, the set
formed by all the intersections of complex lines with the sphere, is the set
of fibers of a fibration of $S^{3}$ by oriented great circles, which is
known as the \textbf{standard Hopf fibration}. Any fibration congruent to it
by an isometry of $S^{3}$ (which does not necessarily preserve the
orientation) is called a \textbf{Hopf fibration}.

The following proposition is known, for instance, from \cite{GW}. For the
reader's convenience we give a proof in the framework on this subsection.

\begin{proposition}
	\label{HopfFactor}A subset $A$ of $\mathcal{C}$ consists of all the fibers
	of a Hopf fibration if and only if $\Phi \left( A\right) =S^{2}\times
	\left\{ z\right\} $ or $\Phi \left( A\right) =\left\{ z\right\} \times S^{2}$
	for some $z\in S^{2}$.
\end{proposition}

\begin{proof}
	As above, let $c_{o}=\left[ s\mapsto e^{is}\right] $. The standard Hopf
	fibration has fibers $c_{o}q$, with $q\in S^{3}$. By (\ref{fyF}), the
	elements of $O\left( 4\right) $ have either the form $q\mapsto p_{1}q%
	\overline{p_{2}}$ (those preserving orientation) or the form $q\mapsto p_{1}%
	\overline{q}\,\overline{p_{2}}$ (those inverting orientation), with $p_{1}$, 
	$p_{2}\in S^{3}$. Then, the set of fibers of a Hopf fibration has the form $%
	H_{l}$ or $H_{r}$, where 
	\begin{equation*}
		H_{l}=\left\{ p_{1}c_{o}q\overline{p_{2}}\mid q\in S^{3}\right\} \text{\ \ \
			\ \ \ and \ \ \ \ \ }H_{r}=\left\{ p_{1}\,\overline{c_{o}q}\,\overline{p_{2}}
		\mid q\in S^{3}\right\} \text{.}
	\end{equation*}
	Now, 
	\begin{equation*}
		H_{l}=\left\{ p_{1}c_{o}\overline{q}\mid q\in S^{3}\right\} =\left\{ \left(
		p_{1},q\right) \cdot c_{o}\mid q\in S^{3}\right\}
	\end{equation*}
	and hence $\Phi \left( H_{l}\right) =\left\{ \left( p_{1}i\overline{p_{1}}%
	,qi \overline{q}\right) \mid q\in S^{3}\right\} =\left\{ z\right\} \times
	S^{2}$, with $z=p_{1}i\overline{p_{1}}$, since $q\mapsto f\left( q\right)
	\left( i\right) $ is onto $S^{2}$. On the other hand, we have that $%
	\overline{c_{o}} =\left[ s\mapsto e^{-is}\right] =jc_{o}\left( -j\right) $
	and so 
	\begin{equation*}
		p_{1}\,\overline{c_{o}q}\,\overline{p_{2}}=p_{1}\overline{q}\,\overline{c_{o}%
		}\, \overline{p_{2}}=p_{1}\overline{q}jc_{o}\left( -j\right) \overline{p_{2}}%
		=p_{1} \overline{q}jc_{o}\overline{p_{2}j}=\left( p_{1}\overline{q}%
		j,p_{2}j\right) \cdot c_{o} \text{.}
	\end{equation*}
	Proceeding as for $H_{l}$, we have then that $\Phi \left( H_{r}\right)
	=S^{2}\times \left\{ z\right\} $ with $z=-p_{2}i \overline{p_{2}}$.
\end{proof}


\subsection{Proofs of the results of this section}

\begin{proposition}
	\label{NotControl}For $\kappa =0,1$ and $\alpha ^{2}=\kappa $, the system $%
	\left( \mathcal{A}_{\kappa }^{\alpha },\iota _{\kappa }^{\alpha }\right) $
	is not controllable. Moreover, either if $\kappa =0$ and $\alpha =0$, or if $%
	\kappa =1$ and $\alpha =\pm 1$, a piecewise $\alpha $-admissible curve in $%
	\mathcal{G}_{\kappa }$ consists of parallel straight lines or of great
	circles in a Hopf fibration, respectively.
\end{proposition}

\begin{proof}
	First we consider the Euclidean case with $\alpha =0$. Let $t\mapsto \ell
	_{t}$ be a $0$-admissible curve in $\mathcal{L}=\mathcal{G}_{0}$. For each $%
	t $ there exist $p_{t}\in \ell _{t}$ and $A_{t}$ such that $\frac{d}{dt}\ell
	_{t}=\Gamma _{t}^{\prime }\left( 0\right) $, with $\Gamma _{t}=\Gamma _{\ell
		_{t},p_{t},A_{t}}$, that is, 
	\begin{equation*}
		\Gamma _{t}\left( \tau \right) =\left[ s\mapsto \phi _{\ell
			_{t},p_{t},A_{t}}\left( s,\tau \right) =p_{t}+\tau A_{t}+sv_{t}\right] \text{
			,}
	\end{equation*}
	where $v_{t}\in S^{2}$ is the direction of $\ell _{t}$, in particular, $%
	v_{t}\bot A_{t}$. Via the diffeomorphism $\psi :TS^{2}\rightarrow \mathcal{L}
	$ in (\ref{difeoTS2L}) and recalling the expression for its inverse
	given afterwards, we have 
	\begin{equation*}
		\psi ^{-1}\left( \ell _{t}\right) =\left( v_{t},p_{t}-\left\langle
		p_{t},v_{t}\right\rangle v_{t}\right) \text{ \ \ \ \ \ and\ \ \ \ \ \ }\psi
		^{-1}\Gamma _{t}\left( \tau \right) =\left( v_{t},p_{t}+\tau
		A_{t}-\left\langle p_{t},v_{t}\right\rangle v_{t}\right) \text{.}
	\end{equation*}
	
	Now $\frac{d}{dt}\ell _{t}=\left. \frac{d}{d\tau }\right\vert _{0}\Gamma
	_{t}\left( \tau \right) $ implies that $\left. \frac{d}{dt}\right\vert
	_{0}\psi ^{-1}\left( \ell _{t}\right) =\left. \frac{d}{d\tau }\right\vert
	_{0}\left( \psi ^{-1}\Gamma _{t}\right) \left( \tau \right) $. Comparing the
	first coordinates we obtain $v_{t}^{\prime }=0$. Therefore the curve $\ell
	_{t}$ consists of parallel lines and in particular the system is not controllable.
	
	In order to deal with the spherical case we use the identification $\mathcal{%
		\ C}\cong S^{2}\times S^{2}$ introduced in (\ref{Phi}). Suppose that $\alpha
	=1 $ and let $\gamma =\left( \gamma _{1},\gamma _{2}\right) $ be a piecewise
	admissible curve in $S^{2}\times S^{2}$. Then, the velocity $\gamma ^{\prime
	}\left( t\right) $ of each piece of $\gamma \ $is in the fiber of $\mathcal{%
		A }_{1}^{1}$ over $\gamma \left( t\right) $, which by (\ref{fibrak1}) is
	included in $T_{\gamma _{1}\left( t\right) }S^{2}\times \left\{ 0_{\gamma
		_{2}\left( t\right) }\right\} $. Thus, $\gamma _{2}^{\prime }=0$ and then $%
	\gamma _{2}$ is constant, say $\gamma _{2}\equiv y_{o}$. So, the curve $%
	\gamma $ lies in $S^{2}\times \left\{ y_{o}\right\} $, that consists of the
	fibers of a Hopf fibration, as we saw in Proposition \ref{HopfFactor}. Hence, 
	two oriented circles cannot be joined by a piecewise $1$-admissible 
	curve if they do not share the projection onto the second
	factor. So, the system is not controllable. If $\alpha =-1$ a similar
	argument applies, involving $\left\{ x_{o}\right\} \times S^{2}$.
\end{proof}

\begin{proposition}
	\label{sustancial} Let $\kappa \in\left\{0,1,-1\right\}$. For any $\ell \in 
	\mathcal{G} _{\kappa }$, the fiber of $\mathcal{A}_{\kappa }^{\alpha }$ over 
	$\ell $ is a substantial submanifold of $T_{\ell }\mathcal{G}_{\kappa }$ if
	and only if $\alpha ^{2}\neq \kappa $.
\end{proposition}

\begin{proof}
	Recall that a submanifold $N$ of a vector space $W$ is said to be
	substantial if it is not included in any proper affine subspace of $W$. If $N
	$ is central symmetric, that is $-N=N$, we can substitute subspace for
	affine subspace, since the segment joining two opposite vectors in $N$
	contains the origin. If $W$ is additionally endowed with an inner product $%
	\left\langle ,\right\rangle $, then $N$ is substantial if and only if $%
	\left\langle q,u\right\rangle =0$ for every $q\in N$ only when $u=0$.
	
	Now we prove the statement of the proposition. By homogeneity, we may
	suppose that $\ell =\ell _{o}$. By Proposition \ref{hachek} (a) and (b), it
	suffices to show that \emph{Ad}$\left( K_{\kappa }\right) \left( \xi
	_{\alpha }\right) $ is not contained in a proper subspace of $\mathfrak{p}%
	_{\kappa }$. On this vector space we consider the inner product 
	\begin{equation*}
		\left\langle Z\left( X,Y\right) ,Z\left( U,V\right) \right\rangle
		=\left\langle X,U\right\rangle +\left\langle Y,V\right\rangle 
	\end{equation*}%
	(see (\ref{zeta})). Let $\zeta =Z\left( 
	\begin{array}{cc}
		x & z \\ 
		w & y%
	\end{array}%
	\right) \in \mathfrak{p}_{\kappa }$ and define $f_{\zeta }:\mathbb{R}%
	^{2}\rightarrow \mathbb{R}$ by 
	\begin{equation*}
		f_{\zeta }\left( s,t\right) =\left\langle \text{Ad}\left( k\left( s,t\right)
		\right) \left( \xi _{\alpha }\right) ,\zeta \right\rangle \text{,}
	\end{equation*}%
	where $k\left( s,t\right) \in K_{\kappa }$ is as in (\ref{isotropiaDeG}). 
	
	Suppose that $f_{\zeta }\equiv 0$. Then $\frac{\partial f_{\zeta }}{\partial
		s}\equiv \frac{\partial f_{\zeta }}{\partial t}\equiv 0$ holds and a
	straightforward computation using (\ref{Adk}) gives 
	\begin{align*}
		\tfrac{\partial f_{\zeta }}{\partial s}\left( s,0\right) & =\cos s\ \left(
		\alpha y-x\right) +\sin s\ \left( -\alpha z-w\right) =0\text{,} \\
		\tfrac{\partial f_{\zeta }}{\partial t}\left( s,0\right) & =\cos s\ \left(
		\kappa y-\alpha x\right) +\sin s\ \left( -\alpha w-\kappa z\right) =0\text{.}
	\end{align*}
	
	By the linear independence of $\cos $ and $\sin $ we obtain the linear
	system 
	\begin{equation*}
		\alpha y-x=0\text{, \ \ \ }\kappa y-\alpha x=0\text{,\ \ \ \ }-\alpha z-w=0 
		\text{, }\ \ \ -\alpha w-\kappa z=0\text{.}
	\end{equation*}
	Now, if $\alpha ^{2}\neq \kappa $, the system has only the trivial solution
	and so $\zeta =0$. Thus, in this case, the submanifold is substantial.
	
	Finally, the submanifold is not substantial if $\alpha ^{2}=\kappa $, since
	for $\zeta =Z\left( 
	\begin{array}{cc}
		\alpha & 1 \\ 
		-\alpha & 1%
	\end{array}
	\right) $, a lengthy computation yields $f_{\zeta }\equiv 0$.
\end{proof}

\bigskip

Now we present the proof of the main result.

\medskip

\begin{proof}[Proof of Theorem \protect\ref{TeoP}]
	By Proposition \ref{NotControl}, we have that (a) implies (b) and that the
	last assertion of the theorem is true. The equivalence between (b) and (c)
	was proved in the previous proposition.
	
	Now we verify that (c) implies (a). We apply Sussmann's Orbit Theorem \cite%
	{Sussmann} (we also consulted \cite{Laguna}). We begin by showing the existence
	of a smooth vector field family $D$ defined everywhere whose $D$-orbits are
	the whole manifold. Since $\mathcal{A}_{\kappa }^{\alpha }\rightarrow 
	\mathcal{G}_{\kappa }$ is a fiber bundle with typical fiber $\mathcal{F}%
	_{\kappa }^{\alpha }$ we can take trivializations $U_{i}\times \mathcal{F}%
	_{\kappa }^{\alpha }\rightarrow \pi ^{-1}\left( U_{i}\right) $ ($i\in 
	\mathcal{I}$) in such a way that the union of all $U_{i}$ covers $\mathcal{G}%
	_{\kappa }$. Let 
	\begin{equation}
		D=\left\{ \text{smooth sections }v^{i}:U_{i}\rightarrow \pi ^{-1}\left(
		U_{i}\right) \text{, }i\in \mathcal{I}\right\} \text{,}
		\label{smooth sections}
	\end{equation}%
	which is a smooth vector field family defined everywhere. We have to show
	that its $D$-orbits are the whole manifold.
	
	Let $\Delta _{D}$ be the distribution on $\mathcal{G}_{\kappa }$ defined as
	follows: $\Delta _{D}\left( \ell \right) $ is the subspace of $T_{\ell }%
	\mathcal{G}_{\kappa }$ spanned by all $v\left( \ell \right) $ such that $%
	v\in D$ and $v$ is defined on $\ell $. Since $\alpha ^{2}\neq \kappa $, we
	have by Proposition \ref{sustancial} that $\Delta _{D}\left( \ell \right)
	=T_{\ell }\mathcal{G}_{\kappa }$ for all $\ell $. Then the smallest $D$%
	-invariant distribution containing $\Delta _{D}$ coincides with $T\mathcal{G}%
	_{\kappa }$. By the Orbit Theorem, the $D$-orbit of any $\ell\in\mathcal{G}_{\kappa }$  is the whole $%
	\mathcal{G}_{\kappa }$. 
	
	Finally, notice that if $v\in D$ is as in (\ref%
	{smooth sections}), then $-v$ is also in $D$ by Proposition \ref{hachek}
	(b). This implies that the system is controllable. Indeed, let $\ell
	_{0},\ell ^\prime \in \mathcal{G}_{\kappa }$ and $v^{i}\in D$ ($i=1,\dots ,k$%
	) such that $v_{t_{k}}^{k}\cdots v_{t_{1}}^{1}\left( \ell _{0}\right) =\ell
	^\prime $, where $t\mapsto v_{t}^{i}$ denotes the flow of $v^{i}$. Call $\ell
	_{i}=v_{t_{i}}^{i}\left( \ell _{i-1}\right) $ and suppose that $t_{j}<0$ and 
	$\gamma _{j}:\left[ t_{j},0\right] \rightarrow \mathcal{G}_{\kappa }$ is the
	integral curve of $v^{j}$ with $\gamma _{j}\left( 0\right) =\ell _{j-1}$. If 
	$\gamma^j :\left[ 0,-t_{j}\right] \rightarrow \mathcal{G}_{\kappa }$ is the
	integral curve of $-v^{j}$ with $\gamma^j \left( 0\right) =\ell _{j-1}$, then $%
	\gamma ^j\left( -t_{j}\right) =\ell _{j}$.
\end{proof}

\bigskip

\begin{proof}[Proof of Proposition \protect\ref{fibratrivial}]
	We begin by describing the typical fibers. We consider first the cases $%
	\kappa =0,-1$. Since $\alpha ^{2}\neq \kappa $, we know from the proof of
	Proposition \ref{hachek} that the fiber over $\ell _{o}$ can be identified
	with $K_{\kappa }$. Hence, $\mathcal{F}_{\kappa }^{\alpha }$ is homeomorphic
	to the cylinder by (\ref{isotropiaDeG}). When $\kappa =1$ and $\alpha
	^{2}\neq 1 $, we have by (\ref{fibrak1}) that $\mathcal{F}_{1}^{\alpha }$ is
	homeomorphic to $S^{1}\times S^{1}$.
	
	To see that $\mathcal{A}_{\kappa }^{\alpha }$ and $\mathcal{G}_{\kappa
	}\times \mathcal{F}_{\kappa }^{\alpha }$ are not homeomorphic we show that
	their fundamental groups do not coincide.
	
	First we deal with the cases $\kappa =0,-1$. By Proposition \ref{hachek}, we
	can identify $\mathcal{A}_{\kappa }^{\alpha }=G_{\kappa }$. By (\ref{isomk}%
	), we have 
	\begin{equation*}
		\pi _{1}\left( \mathcal{A}_{0}^{\alpha }\right) =\pi _{1}\left( G_{0}\right)
		=\pi _{1}\left( SO\left( 3\right) \times \mathbb{R}^{3}\right) \text{, \ \ \
			\ }\pi _{1}\left( \mathcal{A}_{-1}^{\alpha }\right) =\pi _{1}\left(
		G_{-1}\right) =\pi _{1}\left( O_{o}\left( 1,3\right) \right) \text{,}
	\end{equation*}
	both equal to $\pi _{1}\left( SO\left( 3\right) \right) =\mathbb{Z}_{2}$. On
	the other hand, $\mathcal{G}_{\kappa }$ is homeomorphic to $TS^{2}$, which$\ 
	$is a deformation retract of $S^{2}$ and in particular, simply connected.
	Thus, 
	\begin{equation*}
		\pi _{1}\left( \mathcal{G}_{\kappa }\times \mathcal{F}_{\kappa }^{\alpha
		}\right) =\pi _{1}\left( TS^{2}\times \mathbb{R}\times S^{1}\right) =\pi
		_{1}\left( S^{1}\right) =\mathbb{Z}\neq \mathbb{Z}_{2}\text{.}
	\end{equation*}
	
	For the case $\kappa =1$ and $\alpha \neq \pm 1$, we know from Proposition %
	\ref{CS2xS2} that $\mathcal{C}$ is diffeomorphic to $S^{2}\times S^{2}$ and
	also that $\mathcal{F}_{1}^{\alpha }=S^{1}\times S^{1}$, by Proposition \ref%
	{CosasEsf} (c). Then $\pi _{1}\left( \mathcal{C}\times \mathcal{F}
	_{1}^{\alpha }\right) =\mathbb{Z}\times \mathbb{Z}$. By Proposition \ref%
	{fibrado}, $\mathcal{A}_{1}^{\alpha }$ is the orbit of $X_{\alpha }=\left(
	\Gamma _{o}^{\alpha }\right) ^{\prime }\left( 0\right) $ by the action of $%
	SO\left( 4\right) $, which is covered by $S^{3}\times S^{3}$ (see (\ref{fyF}
	)). By (\ref{GammaPrima}), $\mathcal{A}_{1}^{\alpha }$ is homeomorphic to $%
	\left( S^{3}\times S^{3}\right) /H$, where $H$ is the isotropy subgroup at $%
	\left( \left( 1+\alpha \right) j,\left( 1-\alpha \right) j\right) \in
	T_{\left( i,i\right) }\left( S^{2}\times S^{2}\right) $. Now, $H$ consists
	of all the elements $\left( p,q\right) \in S^{3}\times S^{3}$ that fix both
	the foot point $\left( i,i\right) $ and $\left( j,j\right) $, since $\alpha
	^{2}\neq 1$. We have that $pi\overline{p}=qi\overline{q}=i$\ and $pj%
	\overline{p}=qj\overline{ q}=j$ if and only if $p=\pm 1$ and $q=\pm 1$. Then 
	\begin{equation*}
		\mathcal{A}_{1}^{\alpha }=S^{3}\times S^{3}/\left\{ \left( \varepsilon
		,\delta \right) :\varepsilon ,\delta =\pm 1\right\} \text{,}
	\end{equation*}
	which is homeomorphic to $\left(S^{3}/\left\{ \pm 1\right\}\right) \times
	\left(S^{3}/\left\{ \pm 1\right\} \right)=\mathbb{R}P^{3}\times \mathbb{R}%
	P^{3}$, whose fundamental group is $\mathbb{Z}_{2}\times \mathbb{Z}_{2}\neq 
	\mathbb{Z}\times \mathbb{Z} $.
\end{proof}

\subsection{Examples\label{Examples}}

In this subsection we give examples of $\alpha $-admissible curves. Since we
have already dealt with various features of the spherical case, we
concentrate on the Euclidean and hyperbolic cases. We relate $\alpha $%
-admissible curves to Jacobi fields and use that to describe all the
homogeneous $\alpha $-admissible curves for $\kappa =0$. This provides
nontrivial examples, which, in their turn, constitute an interesting family
to pose Kendall's problem.

Let $\sigma $ be a unit speed geodesic of $M_{\kappa }$, $\kappa
\in\left\{0,1,-1\right\}$. A \textbf{Jacobi field} along $\sigma $ arises
from geodesic variations as follows: Let $\varphi :\mathbb{R}\times \left(
-\varepsilon ,\varepsilon \right) \rightarrow M_{\kappa }$ be a smooth map
such that for each $t\in \mathbb{R}$, $s\mapsto \varphi \left( s,t\right) =_{%
	\text{def}}\varphi _{t}(s)$ is a unit speed geodesic with $\varphi
_{0}=\sigma $. Then the associated Jacobi field $J$ along $\sigma $ is given
by $J\left( s\right) =\left. \tfrac{d}{dt}\right\vert _{0}\varphi _{t}\left(
s\right)$.

We recall that Jacobi fields are the solutions of the equation $\frac{D^{2}J%
}{dt^{2}}+R^{\kappa }\left( J,\sigma ^{\prime }\right) \sigma ^{\prime }$,
where $R^{\kappa }$ is the curvature tensor of $M_{\kappa }$, given by $%
R^{\kappa }\left( x,y\right) z=\kappa \left( \left\langle z,x\right\rangle
_{\kappa }y-\left\langle z,y\right\rangle _{\kappa }x\right) $ for $x,y,z$
local vector fields on $M_{\kappa }$. We have then that the Jacobi field
along $\sigma $ with initial conditions $J\left( 0\right) =u+a\sigma
^{\prime }\left( 0\right) $ and $\frac{DJ}{dt}\left( 0\right) =v+b\sigma
^{\prime }\left( 0\right) $, with $a,b\in \mathbb{R}$, $u,v\in \sigma
^{\prime }\left( 0\right) ^{\bot }$ turns out to be%
\begin{equation}
	J\left( s\right) =\cos _{\kappa }\left( s\right) U(s)+\sin _{\kappa }\left(
	s\right) V(s)+\left( a+sb\right) \sigma ^{\prime }(s)\text{,}
	\label{JacobiK}
\end{equation}%
where $U,V$ are the parallel fields along $\sigma $ with $U\left( 0\right)
=u $ and $V\left( 0\right) =v$.

The Jacobi fields $J$ arising from \emph{unit speed} geodesic variations are
exactly those with $\frac{DJ}{dt}\bot \sigma ^{\prime }$ (or equivalently,
with $b=0$ in the expression (\ref{JacobiK})). We call $\mathcal{J}_{\sigma
} $ the vector space consisting of all such Jacobi fields along $\sigma $.
There is a canonical surjective linear morphism 
\begin{equation}
	\mathcal{T}_{\sigma }:\mathcal{J}_{\sigma }\rightarrow T_{[\sigma ]}\mathcal{%
		\ \mathcal{G}_{\kappa }}\text{,}\hspace{1cm}\mathcal{T}_{\sigma }(J)=\left. {%
		\ \tfrac{d}{dt}}\right\vert _{0}[\sigma _{t}]\text{,}  \label{isoT}
\end{equation}%
where $\sigma _{t}$ is any variation of $\sigma $ by unit speed geodesics,
associated with $J$ (see Section 2 in \cite{Hitchin}). The kernel of $%
\mathcal{T}_{\sigma }$ is spanned by $\sigma ^{\prime }$. It is convenient
for us to work with the surjection $\mathcal{T}_{\sigma }$ instead of the
more common isomorphism defined on the space of Jacobi fields along $\sigma $
which are orthogonal to $\sigma ^{\prime }$ (see for instance \cite{SalvaiH}
for the hyperbolic case), because of the geodesic variations appearing in
the examples. By a usual abuse of notation, we sometimes write $J^{\prime }=%
\frac{DJ}{ds}$.

\begin{proposition}
	Fix $\alpha \neq 0$ and let $J\in \mathcal{J}_{\sigma }$ with $J\left(
	0\right) \bot J^{\prime }\left( 0\right) $. If 
	\begin{equation}
		\left\Vert J^{\prime }\left( 0\right) \right\Vert =\left\vert \alpha
		\right\vert \text{\ \ \ }\ \ \ \ \text{and \ \ \ \ \ \ }J^{\prime }\left(
		0\right) =\alpha J\left( 0\right) \times \sigma ^{\prime }\left( 0\right) 
		\text{,}  \label{ecuacionesJ}
	\end{equation}
	then $\mathcal{T}_{\sigma }\left( J\right) $ is $\alpha $-admissible.
	Moreover, the converse is true if $\kappa =0,-1$. See Figure \ref{fig:Jacobi}.
\end{proposition}

\begin{figure}[ht!]
	\centerline{
		\includegraphics[width=3in]
		{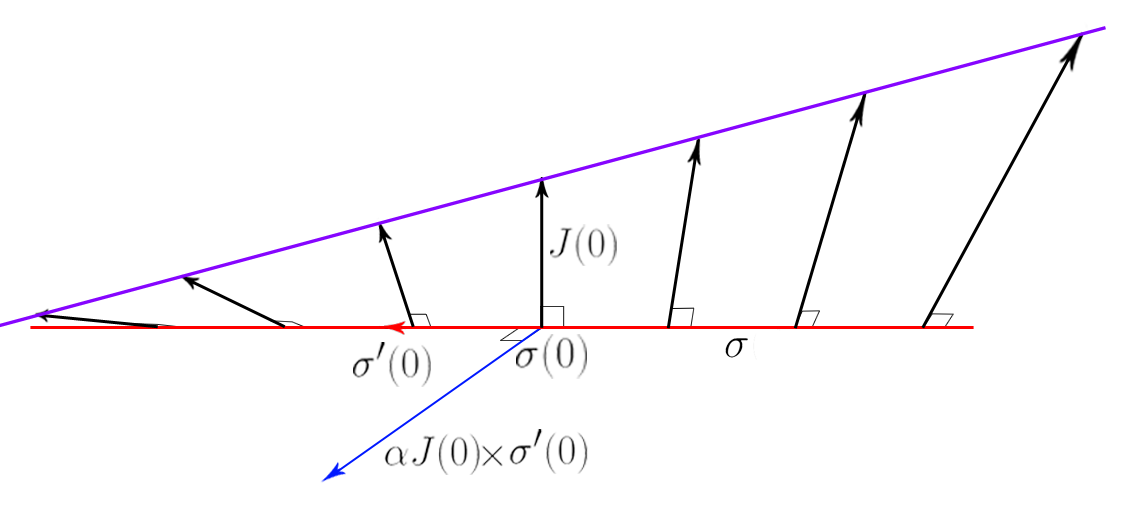}
	}
	\caption{The Jacobi field $J$ in the particular case when $J\left(0\right)$ is perpendicular to 
		$\sigma^\prime\left(0\right)$} \label{fig:Jacobi} 
\end{figure}

\begin{proof}
	Let $\ell =\left[ \sigma \right] $, $p=\sigma \left( 0\right) $ and $%
	A=J\left( 0\right) -\left\langle J\left( 0\right) ,\sigma ^{\prime }\left(
	0\right) \right\rangle \sigma ^{\prime }\left( 0\right) $, which has unit
	norm since 
	\begin{equation*}
		\left\vert \alpha \right\vert =\left\Vert J^{\prime }\left( 0\right)
		\right\Vert =\left\Vert \alpha J\left( 0\right) \times \sigma ^{\prime
		}\left( 0\right) \right\Vert =\left\vert \alpha \right\vert \left\Vert
		A\times \sigma ^{\prime }\left( 0\right) \right\Vert =\left\vert \alpha
		\right\vert \left\Vert A\right\Vert \text{.}
	\end{equation*}
	To prove the first assertion, it suffices to verify that 
	\begin{equation*}
		\mathcal{T}_{\sigma }\left( J\right) =\left. \tfrac{d}{dt}\right\vert _{0} %
		\left[ s\mapsto \phi _{\ell ,p,A}^{\alpha }\left( s,t\right) \right] \text{,}
	\end{equation*}
	or equivalently, that the Jacobi field $L$ along $\sigma $ associated with
	the variation $\phi _{\ell ,p,A}^{\alpha }$ satisfies $\mathcal{T}_{\sigma
	}\left( L\right) =\mathcal{T}_{\sigma }\left( J\right) $. We compute 
	\begin{equation*}
		L\left( 0\right) =\left. \tfrac{d}{dt}\right\vert _{0}\gamma _{\cos \left(
			\alpha t\right) V_{t}+\sin \left( \alpha t\right) B_{t}}\left( 0\right)
		=\left. \tfrac{d}{dt}\right\vert _{0}\gamma _{A}\left( t\right) =A\text{.}
	\end{equation*}
	Also, since $\left. {\frac{D}{ds}}\right\vert _{0}\left. \frac{d}{dt}
	\right\vert _{0}=\left. {\frac{D}{dt}}\right\vert _{0}\left. \frac{d}{ds}
	\right\vert _{0}$, we have that 
	\begin{equation}
		L^{\prime }\left( 0\right) =\left. {\tfrac{D}{dt}}\right\vert _{0}\left. 
		\tfrac{d}{ds}\right\vert _{0}\gamma _{\cos \left( \alpha t\right) V_{t}+\sin
			\left( \alpha t\right) B_{t}}\left( s\right) =\left. {\tfrac{D}{dt}}%
		\right\vert _{0}\cos \left( \alpha t\right) V_{t}+\sin \left( \alpha
		t\right) B_{t}=\alpha B\text{.}  \label{Lprima}
	\end{equation}
	
	On the other hand, $\alpha B=\alpha A\times \sigma ^{\prime }\left( 0\right)
	=\alpha J\left( 0\right) \times \sigma ^{\prime }\left( 0\right) =J^{\prime
	}\left( 0\right) $. Therefore, $L^{\prime }\left( 0\right) =J^{\prime
	}\left( 0\right) $ and $L\left( 0\right) $ differs from $J\left( 0\right) $
	by a multiple of $\sigma ^{\prime }\left( 0\right) $. Thus, $\mathcal{T}
	_{\sigma }\left( L\right) =\mathcal{T}_{\sigma }\left( J\right) $.
	
	Next we prove the converse for $\kappa =0,-1$. We consider $J$ as in (\ref%
	{JacobiK}) with $b=0$ and notice that $J\left( 0\right) =u+a\sigma ^{\prime
	}\left( 0\right) \ \bot \ J^{\prime }\left( 0\right) =v$. Hence $u\bot v$
	and so $U\left( s\right) \bot V\left( s\right) $ for all $s$.
	
	Suppose that $\mathcal{T}_{\sigma }\left( J\right) $ is admissible, that is, 
	$\mathcal{T}_{\sigma }\left( J\right) \in \mathcal{A}_{\kappa }^{\alpha }$.
	Since $\mathcal{T}_{\sigma }\left( J\right) \in T_{\ell }M_{\kappa }$, there
	exist $p\in \ell $ and a unit vector $A\in T_{p}M_{\kappa }$ orthogonal to $%
	\ell $ such that $\mathcal{T}_{\sigma }\left( J\right) =\left. \frac{d}{dt}
	\right\vert _{0}\Gamma _{\ell ,p,A}\left( t\right) $ (here, $p$ and $A$ are
	different from the point and the vector with those names in the first part
	of the proof). Let $s_{o}\in \mathbb{R }$ such that $\sigma \left(
	s_{o}\right) =p$. Putting $\overline{J}\left( s\right) =\left. \frac{d}{dt}%
	\right\vert _{0}\phi _{\ell ,p,A}^{\alpha }\left( s,t\right) $, we have that 
	$s\mapsto \overline{J}\left( s-s_{o}\right) \in \mathcal{J}_{\sigma }$ and
	its image under $\mathcal{T} _{\sigma }$ equals $\mathcal{T}_{\sigma }\left(
	J\right) $. Since $\mathcal{T }_{\sigma }$ is a surjective morphism, $%
	J\left( s\right) =\overline{J}\left( s-s_{o}\right) +c\sigma ^{\prime
	}\left( s\right) $ holds for some $c\in \mathbb{R}$. Then 
	\begin{equation*}
		J\left( s_{o}\right) =\overline{J}\left( 0\right) +c\sigma ^{\prime }\left(
		s\right) =\left. \tfrac{d}{dt}\right\vert _{0}\phi _{\ell ,p,A}^{\alpha
		}\left( 0,t\right) +c\sigma ^{\prime }\left( s_{o}\right) =A+c\sigma
		^{\prime }\left( s_{o}\right) \text{.}
	\end{equation*}
	Similar computations as in (\ref{Lprima}) yield 
	\begin{equation*}
		J^{\prime }\left( s_{o}\right) =\left. \tfrac{D}{ds}\right\vert
		_{s_{o}}\left. \tfrac{d}{dt}\right\vert _{0}\gamma _{\cos \left( \alpha
			t\right) V_{t}+\sin \left( \alpha t\right) B_{t}}\left( s-s_{o}\right)
		=\alpha B_{0}=\alpha A\times \sigma ^{\prime }\left( s_{o}\right) \text{.}
	\end{equation*}
	In particular, $\left\Vert J^{\prime }\left( s_{o}\right) \right\Vert
	=\left\vert \alpha \right\vert $. Therefore, if we show that $s_{o}=0$, then
	both equations in (\ref{ecuacionesJ}) are true. We observe that $J\left(
	s_{o}\right) \bot J^{\prime }\left( s_{o}\right) $. Since we know that $%
	U\bot V$, using expression (\ref{JacobiK}), we have that 
	\begin{eqnarray*}
		0 &=&2\left\langle \cos _{\kappa }\left( s_{o}\right) U(s_{o})+\sin _{\kappa
		}\left( s_{o}\right) V(s_{o}),-\kappa \sin _{\kappa }\left( s_{o}\right)
		U(s_{o})+\cos _{\kappa }\left( s_{o}\right) V(s_{o})\right\rangle \\
		&=&\left( -\kappa \left\Vert u\right\Vert ^{2}+\left\Vert v\right\Vert
		^{2}\right) \sin _{\kappa }\left( 2s_{o}\right) \text{.}
	\end{eqnarray*}
	
	Now, we see that the first factor does not vanish and hence $s_{o}=0$, as
	desired. Indeed, if it were zero, then $\left\Vert v\right\Vert =0$ and so $%
	J^{\prime }\left( s_{o}\right) =-\kappa \sin _{\kappa }\left( s_{o}\right)
	U(s_{o})$. If $\kappa =0$, this implies that $J^{\prime }\left( s_{o}\right)
	=0$. If $\kappa =-1$, then $\left\Vert u\right\Vert =\left\Vert v\right\Vert
	=0$, and so $J^{\prime }\left( s_{o}\right) =0$ as well. In either case we
	have a contradiction, since $\left\Vert J^{\prime }\left( s_{o}\right)
	\right\Vert =\left\vert \alpha \right\vert \neq 0$.
\end{proof}

\bigskip

Next we focus on the Euclidean case. Let $\phi :\mathbb{R}^2\longrightarrow 
\mathbb{R}^{3}$, $\phi \left( s,t\right) =\beta \left( t\right) +sV\left(
t\right) $ be a ruled parametrized surface with $\left\Vert V\right\Vert =1$
which is nowhere cylindrical, that is, $V^{\prime }\left( t\right) \neq 0$
for all $t$. It is said to be \textbf{standard }if $\beta ^{\prime }\bot
V^{\prime }$. It is well-known that every nowhere cylindrical ruled surface
admits such a parametrization; in this case $\beta $ is called the striction
line.

\begin{corollary}
	\label{admisibleConCruz}Let $\alpha\ne 0$, let $\phi :\mathbb{R}
	^2\longrightarrow \mathbb{R}^{3},$ $\phi \left( s,t\right) =\beta \left(
	t\right) +sV\left( t\right) $, be a standard parame\-trized ruled surface and
	let $\Gamma $ be the curve in $\mathcal{L}$ given by $\Gamma \left( t\right)
	=\left[ s\mapsto \phi \left( s,t\right) \right] $. Then $\Gamma ^{\prime
	}\left( 0\right) $ is $\alpha $-admissible if and only if 
	\begin{equation}
		\left\Vert V^{\prime }\left( 0\right) \right\Vert =\left\vert \alpha
		\right\vert \text{\ \ \ }\ \text{and \ \ }V^{\prime }\left( 0\right) =\alpha
		\beta ^{\prime }\left( 0\right) \times V\left( 0\right) \text{.}
		\label{equationsRuled}
	\end{equation}
\end{corollary}

\begin{proof}
	Let $\sigma \left( s\right) =\phi \left( s,0\right) $ and let $J$ be the
	Jacobi field along $\sigma $ associated with the variation $\phi $, that is, $%
	J\left( s\right) =\beta ^{\prime }\left( 0\right) +sV^{\prime }\left(
	0\right) $. Since $\phi $ is standard we have that $J\left( 0\right) =\beta
	^{\prime }\left( 0\right) $ is orthogonal to $J^{\prime }\left( 0\right)
	=V^{\prime }\left( 0\right) $. Now, $\Gamma ^{\prime }\left( 0\right) = 
	\mathcal{T}_{\sigma }\left( J\right) $, and so the assertion is an immediate
	consequence of the previous proposition in the Euclidean case.
\end{proof}

\bigskip

In the next proposition we present the details of Example (c) after
Proposition \ref{fibratrivial}.

\begin{proposition}
	\label{HelCirc}Given $\alpha \neq 0$, let $\varphi $ be the ruled surface
	describing the circular helicoid with radius $r$ and angular velocity $%
	\alpha $, that is, $\varphi \left( s,t\right) =c\left( t\right) +sv\left(
	t\right) $ with 
	\begin{equation*}
		c\left( t\right) =r\left( \cos \left(\tfrac t r\right)e_1 + \sin \left(
		\tfrac t r\right) e_2\right) \text{\ \ \ \ \ and \ \ \ \ }v\left( t\right) =
		\cos \left( \alpha t\right) \tfrac 1 r\, c(t) + \sin \left( \alpha t\right)
		e_3 \text{, }
	\end{equation*}
	and let $\Gamma \left( t\right) =\left[ s\mapsto \varphi \left( s,t\right) %
	\right] $ be the associated curve in $\mathcal{L}$. Then $\Gamma ^{\prime
	}\left( 0\right) $ is not $\alpha $-admissible.
\end{proposition}

\begin{proof}
	Since $\varphi $ is nowhere cylindrical, it admits a standard
	parametrization $\psi \left( s,t\right) =\beta \left( t\right) +sv\left(
	t\right) $ whose associated curve in $\mathcal{L}$ is $\Gamma$. By the Lemma
	above, we have then that $\left\Vert v^{\prime }\left( 0\right) \right\Vert
	=|\alpha |$ is a necessary condition for $\Gamma ^{\prime }\left( 0\right) $
	to be $\alpha $-admissible. But $\left\Vert v^{\prime }\left( 0\right)
	\right\Vert ^{2}=\left\Vert \left( 1/ r\right)e_2+\alpha e_3\right\Vert
	^{2}\allowbreak=\alpha ^{2}+1/r^{2}$. Then, $\Gamma ^{\prime }\left( 0\right) $ is not $%
	\alpha $-admissible.
\end{proof}

\bigskip

Now we characterize the $\alpha $-admissible homogeneous curves in $\mathcal{%
	\ L}$, that is, those which are orbits of monoparametric groups of rigid
transformations. We exclude the trivial case $\alpha =0$. For $s\in \mathbb{R%
}$, let $R_{s}$ be the rotation through the angle $s$ around the $z$-axis
and $T_{s}$ the translation given by $T_{s}\left( x\right) =x+se_{3}$.

\begin{proposition}
	\label{homog}\emph{a)} Any homogeneous curve in $\mathcal{L}$ is congruent,
	via an orientation preserving isometry, to the orbit under the one parameter
	group $t\mapsto R_{\theta t}T_{\lambda t}$ \emph{(}for some $\theta ,\lambda 
	$\emph{)} of the oriented line 
	\begin{equation}
		\ell =\left[ s\mapsto \rho e_{2}+s\left( \sin \eta \ e_{1}+\cos \eta \
		e_{3}\right) \right] \text{,}  \label{ele=}
	\end{equation}
	for some $\rho \geq 0$ and $\eta $.
	
	\medskip
	
	\emph{b)} Let $\alpha \neq 0$. Then the curve $\Gamma $ in $\mathcal{L}$
	given by $\Gamma \left( t\right) =R_{\theta t}T_{\lambda t}\ell $ is $\alpha 
	$-admissible if and only if 
	\begin{equation}
		\left\vert \theta \sin \eta \right\vert =\left\vert \alpha \right\vert \text{
			\ \ \ \ and \ \ \ }\alpha \left( \lambda +\rho \theta \cot \eta \right)
		=\theta \text{.}  \label{equationsH}
	\end{equation}
\end{proposition}

For instance, for $\rho =0$, $\eta =\pi /2$, $\theta =\alpha $ and $\lambda
=1$, we have that $\Gamma =\Gamma _{o}^{\alpha }$ as in (\ref{gamcer}).
Also, for $\rho >0$, $\lambda =0$ and $\theta ,\eta $ related by the
equations, $\Gamma $ is an $\alpha $-admissible curve sweeping a hyperboloid
of one sheet.


\begin{proof}
	a) Let $t\mapsto g_{t}$ be a monoparametric group of rigid transformations
	of $\mathbb{R}^{3}$. It is well known that there exist $\theta $, $\lambda $
	and $h\in G_{0}$ such that $g_{t}=hR_{\theta t}T_{\lambda t}h^{-1}$ for all $%
	t$. Given$\ \ell ^{\prime }\in \mathcal{L}$, we can find $f\in G_{0}$
	commuting with $R_{\theta t}T_{\lambda t}$ such that $f^{-1}\left(
	h^{-1}\ell ^{\prime }\right) $ is $\ell $ as in (\ref{ele=}) for some $%
	\rho\geq 0, \eta $. Then, $t\mapsto g_{t}\ell ^{\prime }=hR_{\theta
		t}T_{\lambda t}h^{-1}\ell ^{\prime }=hR_{\theta t}T_{\lambda t}f\ell
	=hfR_{\theta t}T_{\lambda t}\ell $, as desired.
	
	\medskip
	
	b) We have that $\Gamma \left( t\right) =\left[ s\mapsto \phi \left(
	s,t\right) \right] $ with 
	\begin{equation*}
		\phi \left( s,t\right) =R_{\theta t}\left( \rho e_{2}+s\left( \sin \eta \
		e_{1}+\cos \eta \ e_{3}\right) \right) +t\lambda e_{3}=\beta \left( t\right)
		+sV\left( t\right) \text{,}
	\end{equation*}
	where $\beta \left( t\right) =\rho R_{\theta t}e_{2}+t\lambda e_{3}$ and $%
	V\left( t\right) =R_{\theta t}\left( \sin \eta \ e_{1}+\cos \eta \
	e_{3}\right) $. We may suppose that $\theta \sin \eta \neq 0,$ since
	otherwise, on the one hand, equations (\ref{equationsH}) do not hold and on
	the other hand, the orbit of $\ell $ sweeps either a plane or a cylinder and
	so it not $\alpha $-admissible for $\alpha \neq 0$. Straightforward
	computations yield that $\phi $ is a standard parametrized ruled surface, $%
	\beta ^{\prime }\left( 0\right) =\lambda e_{3}-\rho \theta e_{1}$ and $%
	V^{\prime }\left( 0\right) =\theta \sin \eta \ e_{2}$.
	
	In order to apply Corollary \ref{admisibleConCruz}, we compute $\left\Vert
	V^{\prime }\left( 0\right) \right\Vert =\left\vert \theta \sin \eta
	\right\vert $. Also, the equation $\alpha \beta ^{\prime }\left( 0\right)
	\times V\left( 0\right) =V^{\prime }\left( 0\right) $ translates into $%
	\alpha \left( \lambda e_{3}-\rho \theta e_{1}\right) \times \left( \sin \eta
	\ e_{1}+\cos \eta \ e_{3}\right) =\theta \sin \eta \ e_{2}$, or
	equivalently, 
	\begin{equation*}
		\alpha \left( \lambda \sin \eta +\rho \theta \cos \eta \right) e_{2}=\theta
		\sin \eta \ e_{2}\text{.}
	\end{equation*}
	Therefore, by the corollary, $\Gamma ^{\prime }\left( 0\right) $ is $\alpha $
	-admissible if and only if equations (\ref{equationsH}) hold. By the
	homogeneity of $\Gamma $ and $\mathcal{A}_{0}^{\alpha }$, this is equivalent
	to $\Gamma ^{\prime }\left( t\right) $ being $\alpha $-admissible for all $t$%
	.
\end{proof}

\section{Kendall's problem for some families of $\protect\alpha $-admis\-sible
	curves\label{Kendall}}

This section addresses the analogue mentioned in the introduction of the
well known rolling Kendall's problem. Given a family $\mathcal{F}$ of curves
in a smooth manifold $N$, the \textbf{Kendall number }of $\mathcal{F}$ is
the minimum number of pieces in $\mathcal{F}$ of continuous curves in $N$
taking an initial point to a final point in $N$, both arbitrary and
different.

We consider $N=\mathcal{G}_{0}=\mathcal{L}$ and two families of
distinguished $\alpha $-admissible curves there: the family $\mathcal{P}%
^{\alpha }$, consisting of all (pure) $\alpha $-helicoidal curves, that is,
all curves $\Gamma _{\ell ,p,A}^{\alpha }$ as in Definition \ref{def}, and
the family $\mathcal{H}^{\alpha }$ of all the $\alpha $-admissible
homogeneous curves in $\mathcal{L}$. Note that this renders the result in
Theorem \ref{TeoP} supefluous in the Euclidean case.

In the original Kendall's problem of a sphere rolling on the plane without
slipping and spinning, the most difficult case was to roll along successive
straight lines from a given position to another one over the same point, but
rotated through some angle. In our problem, the most complex case will be to
reach $-\ell $ from $\ell $, two lines with the same image and opposite
directions.

\subsection{Kendall's problem for the family $\mathcal{P}^{\protect\alpha }$}

\begin{proposition}
	\label{Kendall3}For $\alpha \neq 0$, the Kendall number of the family $%
	\mathcal{P}^{\alpha }$ is 3.
\end{proposition}

We begin by stating the following proposition, that implies that this number
is greater than or equal to 3.

\begin{proposition}
	\label{nicon1nicon2}Given $\alpha \neq 0$, the oriented straight lines $\ell 
	$ and $-\ell $ cannot be connected by a continuous curve of two 
	$\alpha$-helicoidal pieces.
\end{proposition}

\begin{proof}
	Without loss of generality, we may suppose that $\ell =\left[ s\mapsto
	se_{1} \right] $ (so, $-\ell =\left[ s\mapsto -se_{1}\right] $) and that the
	first piece is $\Gamma _{\ell ,0,e_{2}}$, defined on the interval $\left[
	0,t_{0} \right] $. We call $\ell _{1}=\Gamma _{\ell ,0,e_{2}}\left(
	t_{0}\right)\ne -\ell$. We denote by $v$ the direction of $\ell _{1}$, which is
	orthogonal to $e_{2}$ (the direction of the axis of the first piece), so we
	can write $v=xe_{1}+ze_{3}$ with $x^2+z^2=1$.
	
	Now we assume that there exist $p\in \ell _{1}$, a unit vector $A$
	orthogonal to $\ell _{1}$ and $t_{1}$ such that $\Gamma _{\ell
		_{1},p,A}\left( t_{1}\right) =-\ell $. Since the axis of $\Gamma _{\ell
		_{1},p,A}$ is orthogonal to $-\ell $ and $\ell _{1}$, we have that $%
	\left\langle A,e_{1}\right\rangle =0=\left\langle A,v\right\rangle $. Hence $%
	z\left\langle A,e_{3}\right\rangle =0$.
	
	If $z=0$, then $v=\pm e_{1}$ and so $p=t_{0}e_{2}+s_{0}e_{1}$ for some $%
	s_{0} $. Since the axis $t\mapsto p+tA$ of $\Gamma _{\ell _{1},p,A}$
	intersects $-\ell $ at $t_{1}$, we have that $p+t_{1}A=s_{0}^{\prime }e_{1}$
	for some $s_{0}^{\prime }$. Now, 
	\begin{equation*}
		s_{o}=\left\langle t_{0}e_{2}+s_{0}e_{1},e_{1}\right\rangle =\left\langle
		-t_{1}A+s_{0}^{\prime }e_{1},e_{1}\right\rangle =s_{0}^{\prime }\text{.}
	\end{equation*}
	Then there exists $\varepsilon =\pm 1$ such that $A=\varepsilon e_{2}$ and $%
	t_{1}=-\varepsilon t_{0}$ and thus $\Gamma _{\ell _{1},p,A}$ travels the
	same path as $\Gamma _{\ell ,s_{0}e_{1},e_{2}}$ if $\varepsilon =1$ or
	backwards if $\varepsilon =-1$. Therefore, $\Gamma _{\ell _{1},p,A}\left(
	t_{1}\right) =\ell \neq -\ell $. If $\left\langle A,e_{3}\right\rangle =0$,
	then $A=\pm e_{2}$, a situation we have already considered.
\end{proof}

\smallskip

\begin{proof}[Proof of Proposition \protect\ref{Kendall3}]
	We know from the previous proposition that the Kendall number of $\mathcal{P}%
	^{\alpha }$ is greater than or equal to 3. Given $\ell $ and $\ell ^{\prime }
	$ in $\mathcal{L}$\textbf{, }we want to achieve $\ell ^{\prime }$ from $\ell 
	$ via the juxtaposition of three $\alpha $-helicoidal curves in $\mathcal{L}$%
	. Without loss of generality we may assume that $\ell ^{\prime }=\left[
	s\mapsto se_{1}\right] $ and $\ell =\left[ s\mapsto de_{2}+sv\right] $ for
	some $d\geq 0$ and some unit vector $v$ orthogonal to $e_{2}$. We consider
	first the case $\alpha >0$.
	
	Let $\Gamma _{1}=\Gamma _{\ell,de_{2},e_{2}}^{\alpha }$, that is, the $%
	\alpha $-helicoidal curve with initial ray $\ell$ and axis parting from $%
	de_{2}$ with direction $e_{2}$. Let $y_{1}\left( t\right) e_{2} $ be the
	point where $\Gamma _{1}\left( t\right) $ intersects the $y$ -axis. Let $%
	t_{1}>0$ be such that the direction of $\ell _{1}=_{\text{def} }\Gamma
	_{1}\left( t_{1}\right) $ is $-e_{3}$ and $y_{1}\left( t_{1}\right) > \frac{%
		\pi }{2\alpha }$. See Figure \ref{threrLines}
	
	\begin{figure}
		\centering
		\begin{minipage}{0.45\textwidth}
			\centering
			\includegraphics[width=0.7\textwidth]{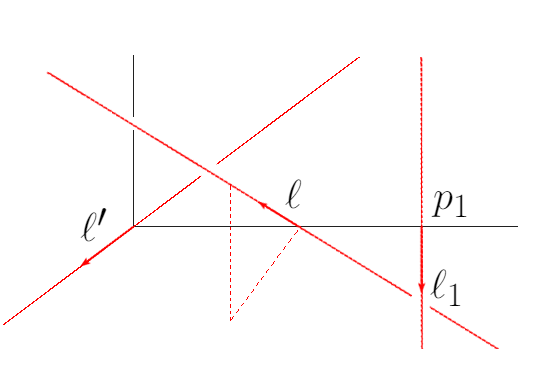} 
			\caption{The lines $\ell$, $\ell^\prime$ and $\ell_1$ intersecting the  vertical plane $x=0$}
			\label{threrLines}
		\end{minipage}\hfill
		\begin{minipage}{0.45\textwidth}
			\centering
			\includegraphics[width=0.7\textwidth]{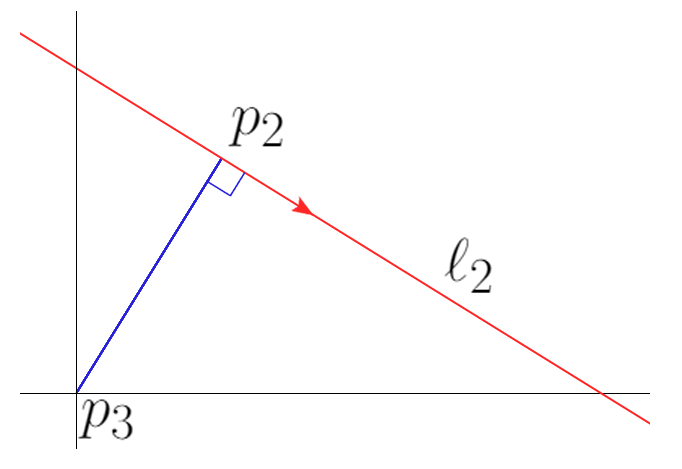} 
			\caption{The line $\ell_2$ in the plane $x=t_2$; $\left\Vert
				p_{3}-p_{2}\right\Vert=\frac {\pi}{2\alpha}$ }
			\label{triangle}
		\end{minipage}
	\end{figure}
	Let $\Gamma _{2}=\Gamma _{\ell _{1},p_{1},e_{1}}^{\alpha }$ where $%
	p_{1}=y_{1}\left( t_{1}\right) e_{2}$. For each $t$ we consider the distance 
	$f\left( t\right) $ between $\Gamma _{2}\left( t\right) $ and $\ell^{\prime
	} $. We have that $f\left( 0\right) =y_{1}\left( t_{1}\right) $. By the
	continuity of $f$, if $\tau $ is the first positive zero of $f$, there
	exists $0\leq t_{2}<\tau $ such that $f\left( t_{2}\right) =\frac{\pi }{%
		2\alpha }$.
	
	Call $\ell _{2}=\Gamma _{2}\left( t_{2}\right) $ and let $p_{2}$ and $p_{3}$
	be the points in $\ell _{2}$ and $\ell^{\prime } $, respectively, realizing
	the distance between these lines. Let $A=\frac{p_{3}-p_{2}}{\left\Vert
		p_{3}-p_{2}\right\Vert }$ and $\Gamma _{3}=\Gamma _{\ell
		_{2},p_{2},A}^{\alpha }$. Then $\Gamma _{3}\left( \frac{\pi }{2\alpha }
	\right) =\ell ^{\prime }$, since $\frac{\pi }{2\alpha }$ is the time an $%
	\alpha $ -helicoidal curve takes to make one fourth of a complete turn. See Figure \ref{triangle}.

	If $\alpha <0$, similar arguments hold, setting the direction of $\ell _{1}$
	equal to $e_{3}$ and substituting $\frac{\pi }{2\alpha }$ with $\frac{\pi }{
		2\left\vert \alpha \right\vert }$.
\end{proof}

\subsection{Kendall's problem for the family $\mathcal{H}^{\protect\alpha }$}

The elements of the family $\mathcal{H}^{\alpha }$ of all $\alpha $%
-admissible homogeneous curves in $\mathcal{L}$ for $\alpha \neq 0$ have
been described in Proposition \ref{homog}.

\begin{proposition}
	Let $\alpha \neq 0$. The Kendall number of the family $\mathcal{H}^{\alpha }$
	is $2$.
\end{proposition}

\begin{proof}
	First of all, we check that two intersecting lines $\ell $ and $\ell
	^{\prime }$, with $\ell ^{\prime }\neq \pm \ell $, can be joined by one
	curve in $\mathcal{H}^{\alpha }$. If they form an angle $0<2\eta <\pi $, we
	may suppose without loss of generality that 
	\begin{equation*}
		\ell =\left[ s\mapsto s\left( \sin \eta ,0,\cos \eta \right) \right] \text{\
			\ \ \ \ and\ \ \ \ \ }\ell ^{\prime }=\left[ s\mapsto \left( 0,0,\tfrac{\pi 
		}{\alpha }\right) +s\left( -\sin \eta ,0,\cos \eta \right) \right] \text{.}
	\end{equation*}%
	\begin{figure}[ht!]
		\label{fig:Poposition20} 
		\centerline{
			\includegraphics[width=1.5in]
			{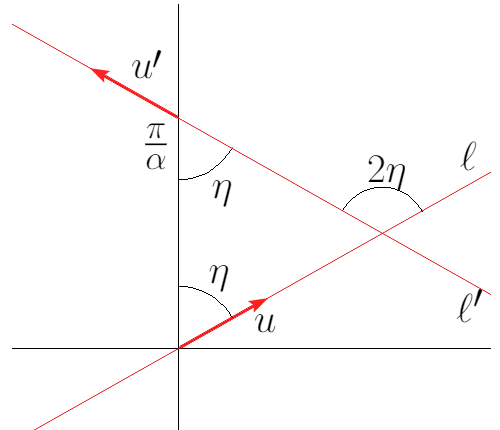}
		}
		\caption{Standard position of $\ell$ and $\ell^\prime$ when they intersect. Also, 
			$(dh_{\pi/\theta})_{0}(u)=(u^\prime)$}
		\label{linesIntersect}
	\end{figure}
	Let $\Gamma $ be the curve in $\mathcal{L}$ determined by the orbit of $\ell 
	$ under the monoparametric group $h_t=_{\text{def}}R_{\theta t}T_{\lambda t}$ as in
	Proposition \ref{homog}, with $\theta =\frac{\alpha }{\sin \eta }$ and $%
	\lambda =\frac{1}{\sin \eta }$ ($\rho =0$). The curve is $\alpha$-admissible 
	since the corresponding equations given in (\ref{equationsH})
	are satisfied. One can also verify easily that $\Gamma \left( \frac{\pi }{%
		\theta }\right) =\ell ^{\prime }$. Thus, $\ell $ and $\ell ^{\prime }$ can
	be joined by one $\alpha $-admissible homogeneous curve. See Figure \ref{linesIntersect}.
	
	Now we consider two lines $\ell $ and $\ell ^{\prime }$ in $\mathcal{L}$
	that do not intersect. Let $\ell _{1}\in \mathcal{L}$ containing the
	shortest segment joining $\ell $ to $\ell ^{\prime }$, which is
	perpendicular to both of them. By the case above with $2\eta =\frac{\pi }{2}$
	, $\ell ^{\prime }$ can be reached from $\ell $ via the juxtaposition of two
	curves in $\mathcal{H}^{\alpha }$, the first joining $\ell $ to $\ell _{1}$
	and the second joining $\ell _{1}$ to $\ell ^{\prime }$. If $\ell ^{\prime
	}=-\ell $, one can take as $\ell _{1}$ any curve orthogonal to $\ell $.
	Then, the Kendall number is at most $2$.
	
	Finally, we show that the Kendall number is greater than 1. It suffices to
	see that for the monoporametric group $t\mapsto g_{t}=R_{\theta t}T_{\lambda
		t}$ as in Proposition \ref{homog}, if $t\mapsto g_{t}\left( \ell \right) $
	is $\alpha $-admissible for some $\ell \in \mathcal{L}$, then $g_{t}\left(
	\ell \right) \neq -\ell $ for all $t$. We may suppose that $\ell =\left[
	s\mapsto \rho e_{2}+sv\right] $ is as in (\ref{ele=}). The direction of $%
	g_{t}\left( \ell \right) $ is $R_{t}\left( v\right) $. If $g_{t}\left( \ell
	\right) =-\ell $, then $R_{t}\left( v\right) =-v$ and equating the third
	components yields $\cos \eta =-\cos \eta $ and so $\cos \eta =0$. In
	particular, $\ell $ is contained in the plane $z=0$. Now, equations (\ref%
	{equationsH}) imply that 
	\begin{equation*}
		\left\vert \theta \right\vert =\left\vert \alpha \right\vert \text{\ \ \ \
			and \ \ \ }\alpha \lambda =\theta \text{.}
	\end{equation*}
	Hence, $g_{t}=R_{\lambda \alpha t}T_{\lambda t}$ with $\lambda =\pm 1$.
	Since $g_{t}\ell $ is contained in the plane $z=\lambda t$, we have that $%
	g_{t}\ell \neq -\ell $ for all $t$ (otherwise, we get $\lambda =0$, a
	contradiction).
\end{proof}

\vspace{0.5cm}

\noindent Mateo Anarella\newline
\noindent KU Leuven, Department of Mathematics\newline
\noindent Celestijnenlaan 200b - box 2400, 3001 Leuven, Belgium\newline
\noindent mateo.anarella@kuleuven.be

\bigskip

\noindent Marcos Salvai\newline
\noindent \textsc{f}a\textsc{maf} (Universidad
Nacional de C\'{o}rdoba) and \textsc{ciem} (Conicet)\newline
\noindent Av. Medina Allende s/n, Ciudad Universitaria, CP:X5000HUA C\'ordoba, Argentina\newline
\noindent marcos.salvai@unc.edu.ar

\end{document}